\documentclass[reqno]{amsart}
\UseRawInputEncoding
\makeatletter
\@namedef{subjclassname@2020}{\textup{2020} Mathematics Subject Classification}
\makeatother

\usepackage{amssymb}
\usepackage[all]{xy}
\usepackage{enumitem}
\usepackage{hyperref}

\def\commentout#1{{}}

\def\Z{{\mathbb Z}}

\def\R{{\mathbb R}}
\def\C{{\mathbb C}}
\def\N{{\mathbb N}}

\def\F{{\mathbb F}}

\def\cC{{\mathcal C}}

\def\profin#1{{{#1}^{{\kern -0.15em}\wedge}}}
\def\Ghat{{\profin{G}}}
\def\Xhat{{\profin{X}}}
\def\Ihat{{\profin{I}}}
\def\Jhat{{\profin{J}}}
\def\Rhat{{\profin{R}}}
\def\fhat{{\profin{f}}}

\def\GL{{\mathrm{GL}}}

\def\bd{{\mathbf{d}}}

\newcommand\Hom{\operatorname{Hom}}

\newcommand\Ker{\operatorname{Ker}}
\newcommand\Img{\operatorname{Im}}

\newcommand\AS{{\operatorname{AS}}}

\newcommand\surj{{\operatorname{surj}}}
\newcommand\csurj{{\operatorname{csurj}}}
\newcommand\Alg{{\operatorname{Alg}}}
\newcommand\ps{{\operatorname{ps}}}
\newcommand\cs{{\operatorname{cs}}}
\newcommand\FinSets{{\operatorname{FinSets}}}
\newcommand\dom{{\operatorname{dom}}}
\newcommand\parto{\rightharpoonup}
\newcommand\id{{\operatorname{id}}}
\newcommand\Mod{{\operatorname{Mod}}}
\newcommand\pr{{\operatorname{pr}}}
\newcommand\Xcheck{{{\check{X}}}}
\newcommand\Zvcheck{{{\check{Z_v}}}}
\newcommand\HS{{{\operatorname{HS}}}}
\newcommand\tr{{{\operatorname{tr}}}}
\newcommand\ttop{{{\operatorname{top}}}}
\newcommand\bbot{{{\operatorname{bot}}}}
\newcommand\avgY{\operatorname{avg}_Y}
\newcommand\avg{\operatorname{avg}}
\newcommand\wt{\operatorname{wt}}
\newcommand\shape{{{\operatorname{shape}}}}
\newcommand\height{{{\operatorname{sumup}}}} 
\newcommand\End{{{\operatorname{End}}}}
\newcommand\sumup{{{\operatorname{sum}}}} 

\newcommand\ev{{\operatorname{ev}}}

\newtheorem{theorem}{Theorem}[section]
\newtheorem{lemma}[theorem]{Lemma}
\newtheorem{proposition}[theorem]{Proposition}
\newtheorem{corollary}[theorem]{Corollary}

\theoremstyle{definition}
\newtheorem{definition}[theorem]{Definition}

\theoremstyle{remark}
\newtheorem{remark}[theorem]{Remark}

\numberwithin{equation}{section}

\usepackage{amsmath}	
\begin{document}
\title[Profinite association schemes]
{Functoriality of Bose-Mesner algebras
and profinite association schemes}

\author{Makoto Matsumoto}
\address{Mathematics Program \\ 
Graduate School of Advanced Science and Engineering\\
Hiroshima University, 739-8526 Japan}
\email{m-mat@math.sci.hiroshima-u.ac.jp}

\author{Kento Ogawa}
\address{Mathematics Program \\ 
Graduate School of Advanced Science and Engineering\\
Hiroshima University, 739-8526 Japan}
\email{knt-ogawa@hiroshima-u.ac.jp}

\author{Takayuki Okuda}
\address{Mathematics Program \\ 
Graduate School of Advanced Science and Engineering\\
Hiroshima University, 739-8526 Japan}
\email{okudatak@hiroshima-u.ac.jp}

\thanks{
The first author is partially supported by JSPS
Grants-in-Aid for Scientific Research
JP18K03213.
The second author is partially supported by
JST SPRING, Grant Number JPMJSP2132.
The third author is partially supported by JSPS
Grants-in-Aid for Scientific Research
JP20K03589, JP20K14310, and JP22H01124.
}

\date{\today}

\begin{abstract}
We show that taking the set of 
primitive idempotents of commutative association schemes
is a functor from 
the category of commutative association schemes with surjective
morphisms to the category of finite sets with 
surjective partial functions. We then consider 
projective systems of commutative association schemes consisting of
surjections (which we call profinite association schemes),
for which Bose-Mesner algebra is defined,
and describe a Delsarte theory on such schemes.
This is another method for generalizing 
association schemes to those on infinite sets, related with
the approach by Barg and Skriganov.
Relation with $(t,m,s)$-nets and $(t,s)$-sequences is studied.
We reprove some of the results of Martin-Stinson from this viewpoint.
\end{abstract}

\maketitle
\section{Introduction}
Association schemes are central objects in algebraic combinatorics,
with many interactions with other areas of mathematics.
It is natural to try to take projective limits of a system 
of association schemes, but there seems to be some obstruction
to have a natural Bose-Mesner algebra.
We show that a surjective morphism between commutative association schemes
behaves well with respect to 
the two product structures of Bose-Mesner algebras and primitive
idempotents, which gives rise to a notion of profinite
association schemes and their Delsarte theory. They 
are closely related with profinite groups. We study
kernel schemes and ordered Hamming schemes introduced
by Martin-Stinson\cite{MARTIN-STINSON} as examples.
We reprove their characterization of $(t,m,s)$-net
as a design in an ordered Hamming scheme,
and give a characterization of $(t,s)$-sequences
in terms of profinite association schemes.

\section{Surjections among association schemes and functoriality}
Let $\N$ denote the set of natural numbers including $0$, and $\N_{>0}$ the set of 
positive integers.
\subsection{Category of association schemes}
Let us recall the notion of association schemes briefly.
See \cite{Bannai}\cite{DELSARTE} for details.
We summarize basic terminologies.
\begin{definition}
Let $X$ be a set. By $\#X$ we denote the cardinality of
$X$. If $X$ has a topology, $C(X)$ denotes the vector space of continuous
functions from $X$ to $\C$. If a topology is not specified,
$X$ is regarded to be a discrete space. For $S\subset X$, we define
the indicator function of $S$ as a function $\chi_S:X \to \{0,1\}$
where $\chi(x)=1$ if and only if $x \in S$.
If $S$ is a singleton $\{x\}$, we denote the indicator function by $\delta_x$.
For a map $f:X \to Y$ between finite 
sets, $f^\dagger:C(Y) \to C(X)$ denotes the pull back: 
$g\in C(Y) \mapsto g\circ f$. It is injective (respectively surjective)
if $f$ is surjective (respectively injective).
By the multiplication of functions, 
$C(X)$ is a unital commutative ring. This multiplication is 
called the Hadamard product. 
For a finite set $X$, the set $C(X\times X)$
is naturally identified with the set of complex valued 
square matrices of size $\#(X)$, and the matrix product
is given by $AB(x,z)=\sum_{y\in X}A(x,y)B(y,z)$. The Hadamard
product $\circ$ is given by the component wise product, namely,
$(A\circ B)(x,z)=A(x,z)B(x,z)$. (Note that $\circ$ may denote
the composition of mappings, but no confusions would occur.)
\end{definition}

\begin{definition}\label{def:assoc-scheme}
Let $X$, $I$ be finite sets, and $R:X\times X \to I$ a surjection.
We call $(X,R,I)$ an association scheme, 
if the following properties are satisfied.
For each $i \in I$, 
$R^{-1}(i)$ may be regarded as a relation on $X$, denoted by $R_i$.
Let $A_i$ be the corresponding adjacency matrix
in $C(X\times X)$.
\begin{enumerate}
 \item There is an $i_0 \in I$ such that $A_{i_0}$
is an identity matrix.
 \item Consider the injection $R^\dagger:C(I) \to C(X\times X)$.
Its image $A_X$ is closed by the matrix product.
 \item $A_X$ is closed under the transpose of $C(X\times X)$.
\end{enumerate}
The algebra $A_X$ with the two multiplications
(the Hadamard product and the matrix product) is called
the Bose-Mesner algebra of $(X,R,I)$.
We may use the same notation
$i_0$ for distinct association schemes.
The number of $1$ in a row of $A_i$ is called 
the $i$-th valency and denoted by $k_i$.
If $A_X$ is commutative with respect to the matrix product, 
then $(X,R,I)$ is said to be commutative.
\end{definition}
In the following, we write simply an ``association scheme $X$''
for an association $(X,R,I)$ by an abuse of language.
We follow MacLane \cite{MACLANE} for the terminologies of the category theory,
in particular, functors, projective systems, and limits.
The association schemes form a category, 
by the following\cite{HANAKI-CAT}\cite{ZIESCHANG}.
\begin{definition}
Let $(X,R,I)$, $(X',R',I')$ be association schemes.
A morphism of association schemes from $X$ to $X'$ is 
a pair of mappings $f:X \to X'$ and $g:I \to I'$ 
such that the following diagram commutes:
\begin{equation}\label{eq:comm}
\begin{array}{ccc}
X \times X & \to &  I \\
f\times f \downarrow \phantom{f \times f} 
&\circlearrowleft & \phantom{g}\downarrow g  \\
X' \times X'& \to &  I'. \\
\end{array}
\end{equation}
It is clear that the surjectivity of $f$ implies 
that of $g$. The morphism $(f,g)$ is said to be 
surjective if $f$ is surjective. 
\end{definition}

\subsection{Functoriality of Bose-Mesner algebra}
\begin{proposition}\label{prop:fiber} 
If $(f,g):X \to X'$ is a surjective morphism of 
association schemes, then for any $x'\in X'$,
the cardinality of the fiber $\#f^{-1}(x')$ is
$\#X/\#X'$.
\end{proposition}
\begin{proof}
The cardinality of the fiber is the summation of 
the valencies of $R_i$ for $i$ with $g(i)=i_0'$, hence 
independent of the choice of $x'$, which implies the result.
This is well-known \cite{ZIESCHANG}, and it holds
under a weaker condition (for unital regular relation partitions, see
\cite{KMO}).  
\end{proof}

\begin{theorem}\label{th:convolution}
$ $

\begin{enumerate}
\item Let $(f,g):(X,R,I) \to (X',R',I')$ be a morphism. 
It induces an injection $C(I') \to C(I)$ and hence $\Psi: A_{X'} \to A_{X}$.
Then $\Psi$ is a morphism of unital rings with respect to the Hadamard product.
\item Suppose that $f$ is surjective. We define the convolution 
product $\bullet_X$ on $A_X$ by 
$$
 (A\bullet_X B)(x,z)=\frac{1}{\#X}\sum_{y\in X}A(x,y)B(y,z)
$$ 
(i.e., the matrix product normalized by the factor of $\#X$).
Then, $A_X$ has a unit $E_X:=\#X A_{i_0}$, 
and
$\Psi$ is a morphism of rings with respect to the convolution product
(which may not map the unit of $A_{X'}$ to the unit of $A_{X}$).

\item The vector space $C(X)$ is a left $A_X$-module
by the following action:
$$
A_X \times C(X) \to C(X), (A, h) \mapsto (A \bullet h)
$$
given by 
$$
(A \bullet h)(x)=\frac{1}{\#X}\sum_{y\in X}A(x,y)h(y),
$$
and the unit $E_X$ acts trivially.
This module structure is compatible with $\Psi$ in the 
sense that the following diagram commutes:
\begin{equation}\label{eq:mod}
\begin{array}{ccc}
A_X \times C(X) & \to &  C(X) \\
\Psi \times f^\dagger \uparrow \phantom{f^\dagger \times f^\dagger} 
&\circlearrowleft & \phantom{g\dagger}\uparrow f^\dagger  \\
A_{X'} \times C(X')& \to &  C(X'). \\
\end{array}
\end{equation}
If we consider $C(X')\subset C(X)$,
then the element $\Psi(E_{X'})\in A_X$ acts on $C(X)$
as a projector $C(X) \to C(X')$.
\end{enumerate}
\end{theorem}
\begin{proof}
For a morphism of finite sets $X \to Y$, the induced map
$C(Y) \to C(X)$ is a morphism of unital rings with respect to 
the Hadamard product. The first statement follows from this and
the commutative diagram (\ref{eq:comm}).
For the second part, take $A,B \in A_X'$. (In fact, one may
take $A,B \in C(X'\times X')$; the following arguments depend
only on the fact that the cardinality of the fiber of $f$ is constant.)
Then
\begin{eqnarray*}
\Psi(A\bullet_{X'} B)(x,z)&=&(A \bullet_{X'} B)(f(x),f(z)) \\
&=& \frac{1}{\#X'}\sum_{y'\in X'}A(f(x),y')B(y',f(z)) \\
&=& \frac{1}{\#X'}\frac{1}{\#f^{-1}(y')}\sum_{y\in X}A(f(x),f(y))B(f(y),f(z))\\
&=& \frac{1}{\#X'}\frac{\#X'}{\#X}\sum_{y\in X}\Psi(A)(x,y))\Psi(B)(y,z)\\
&=& \frac{1}{\#X}\sum_{y\in X}\Psi(A)(x,y))\Psi(B)(y,z)\\
&=& (\Psi(A)\bullet_{X} \Psi(B))(x,z).\\
\end{eqnarray*}
 
For the third part, it is a routine calculation 
to check that $A_X$ acts on $C(X)$ as a unital ring.
For $A\in A_{X'}$ and $h\in C(X')$, the 
compatibility follows from:
\begin{eqnarray*}
f^\dagger((A\bullet h))(x)&=&(A \bullet h)(f(x)) \\
&=& \frac{1}{\#X'}\sum_{y'\in X'}A(f(x),y')h(y') \\
&=& \frac{1}{\#X'}\frac{1}{\#f^{-1}(y')}
    \sum_{y\in f^{-1}(y'), y'\in X'}A(f(x),f(y))h(f(y))\\
&=& \frac{1}{\#X'}\frac{\#X'}{\#X}
    \sum_{y\in Y}\Psi(A)(x,y))(f^\dagger h)(y)\\
&=& \frac{1}{\#X}\sum_{y\in X}\Psi(A)(x,y))(f^\dagger h)(y)\\
&=& (\Psi(A)\bullet (f^\dagger h))(x).
\end{eqnarray*}
The compatibility implies that $\Psi(E_{X'})$ trivially acts on the image 
of $C(X')$ in $C(X)$. For $A\in A_{X'}$ and $h\in C(X)$,
\begin{eqnarray*}
(\Psi(A)\bullet h))(x)&=&\frac{1}{\#X}\sum_{y\in X}\Psi(A)(x,y)h(y) \\
&=&\frac{1}{\#X}\sum_{y\in X}A(f(x),f(y))h(y),
\end{eqnarray*}
which depends only on $f(x)$, and hence $\Psi(A)\bullet h$
lies in $f^\dagger(C(X')) \subset C(X)$.
Hence $\Psi(E_{X'})$ is a projection $C(X)\to C(X')$.
\end{proof}
\begin{definition}
Let $\AS_\surj$ be the category of association schemes with 
surjective morphisms.  
\end{definition}
\begin{corollary}
Let $\Alg_{HC}$ (HC means Hadamard and convolution) be the category of
finite dimensional $\C$-vector spaces $A$ with:
\begin{enumerate}
\item 
one associative multiplication (Hadamard product)
which gives a commutative unital semi-simple $\C$-algebra structure to $A$,
\item
one (possibly non-commutative) associative multiplication
(convolution product)
which gives a unital semi-simple $\C$-algebra structure to $A$.
\end{enumerate}
(Note that 
in this case, semi-simplicity is equivalent to that $A$ 
is a direct product of a finite number of matrix algebras over $\C$).
Morphisms are injective $\C$-linear maps preserving the both two products and
the unit for the Hadamard product (we don't require preservation of 
unit for the convolution product).

Then, the correspondence $X \mapsto A_X$ gives a contravariant functor
from $\AS_\surj$ to $\Alg_{HC}$.
\end{corollary}
\begin{proof} It is proved that $A_X$ has two products. 
Semi-simplicity is well-known, c.f. \cite{ZIESCHANG}. (It is enough to
show that there is no nilpotent ideal, but for any nonzero $A \in A_X$, 
the product with its unitary conjugate 
$A\bullet A^*=\frac{1}{\#X}AA^* $ is not nilpotent, hence $A_X$ 
is semi-simple.) 
Functoriality is easy to check, using Theorem~\ref{th:convolution}.
\end{proof}

\begin{corollary}\label{cor:mod}
Let us consider the category $\Mod$ of $R$-modules, whose 
object is a pair of a unital commutative ring $R$ and 
an $R$-module $M$, and a morphism from $(R,M)$ to $(R',M')$
is a pair of a ring morphism $f:R \to R'$ and 
a $\Z$-module morphism $g:M \to M'$
which makes the following diagram commute:
$$
\begin{array}{ccc}
R \times M & \to &  M \\
f\times g \downarrow \phantom{f\times g}
&\circlearrowleft & \phantom{g}\downarrow g \\
R' \times M'& \to &  M'. \\
\end{array}
$$
Then, the correspondence $X \mapsto (A_X,C(X))$ is
a contravariant functor from $\AS_\surj$ to $\Mod$.
\end{corollary}
\begin{proof}
The correspondence is given in Theorem~\ref{th:convolution}.
It is a contravariant functor, since each of the correspondences
$X\mapsto A_X$, $X\mapsto C(X)$ is a contravariant functor.
\end{proof}

\begin{remark}
In \cite{French}, French constructed a sub-category of association 
schemes, and a covariant functor from it to the category of algebra.
Our construction produces a contravariant functor, which seems to be
of different nature. 
\end{remark}

\subsection{Commutative association schemes and primitive idempotents}
\label{sec:comm}

By semisimplicity and Artin-Wedderburn Theorem, any Bose-Mesner
algebra $A_X$ (with convolution product) is isomorphic to a product of matrix algebras over $\C$.
We assume that $X$ is commutative.
Then, $A_X$ is, as a unital ring, isomorphic to 
a direct product of copies of $\C$. Namely, 
$A_X\cong \C\times \cdots \times\C=:\C^n$ for some $n\in \N$,
and this decomposition is unique as a direct product of rings. 

\begin{proposition}\label{prop:idempotent}
Let $A$ be a ring isomorphic to $\C^n$ (with componentwise multiplication). 
An element $j\in A$ corresponding to an element of $\C^n$ with one coordinate $1$
and the other coordinates $0$ is called a primitive idempotent of $A_X$. 
It is characterized by the idempotent property $j^2=j$, $j\neq 0$ 
(this is equivalent to that the each coordinate is $0$ or $1$ and
at least one $1$ exists)
and 
that there is no nonzero idempotent $j'\neq j$ such that $jj'=j'$ holds 
(this says only one coordinate is $1$).
Any idempotent is uniquely a sum of primitive idempotents.
\end{proposition}

\begin{definition}
Let $\FinSets$ be the category of finite sets.
Let $\FinSets_\ps$ be the category of finite sets
and partial surjective functions.

Recall that for sets $X$ and $Y$, a partial function 
$f$ from $X$ to $Y$ consists of a pair of a subset $\dom(f)\subset X$
(which may be empty)
and a function $f|_{\dom(f)}:\dom(f)\to Y$. Two partial functions
are equal if they have the same domain and the same function
on the domain. 
A partial function is denoted by $f:X \parto Y$.
For $S\subset Y$, $f^{-1}(S):=\{x \in X \mid x\in \dom(f), f(x)\in S\}$.
Composition $g\circ f$ of partial functions $f:X \parto Y$ and
$g:Y \parto Z$ has domain $f^{-1}(\dom(g))$.
A partial function $f$ is surjective if 
$\Img(f):=\{f(x) \mid x\in \dom(f)\}\subset Y$ is $Y$.
\end{definition}

\begin{proposition}
Let $\Alg_{cs}$ be the category of rings isomorphic to $\C^n$ for some $n$,
with injective $\C$-linear ring morphisms which may not preserve the unit.
For $A\in \Alg_{cs}$, we define $J(A)$ as the finite set of primitive idempotents of $A$.
For a finite set $I$, we define $C(I)$ the set of maps from $I$ to $\C$
as defined above. These are contravariant 
functors, and give contra-equivalence
between $\Alg_{cs}$ and $\FinSets_\ps$.
\end{proposition}
\begin{proof}
Note first that the notion of $f:X \parto Y$ is 
the notion of a family of disjoint subsets 
$S_y \subset X \ \ (y\in Y)$, where $S_y=f^{-1}(y)$,
and the surjectivity of $f$ is equivalent to that $S_y \neq \emptyset$ for all
$y\in Y$. 
To show that $J$ is a functor, let $\varphi:A \to B$
be an injective $\C$-algebra homomorphism. For $j \in J(A)$, $\varphi(j)$
is a nonzero idempotent. Thus, it is a non-empty sum of elements of 
$J(B)$, which gives a non-empty subset $S_j$ of $J(B)$ 
by Proposition~\ref{prop:idempotent}.
We construct a partial surjection $J(B) \parto J(A)$ by:
for $j\in J(A)$, elements of $S_j$ is mapped to $j$.
Note that if $j\neq j'\in J(A)$, then $jj'=0$ and hence
$S_j \cap S_{j'}=\emptyset$. 
Since $S_j \neq \emptyset$, 
the partial function is surjective.
To check the functoriality, we 
consider $A \stackrel{\varphi}{\to} B \stackrel{\phi}{\to} C$. 
The construction of partial surjection for $\varphi$
is done by assigning to $j \in J(A)$ the set $S_j \subset J(B)$
such that summation of elements in $S_j$ is $\varphi(j)$.
Similarly, for each $j'\in S_j$, we have $S_{j'}\subset J(C)$.
Thus, $J(\varphi) \circ J(\phi)$ maps $\coprod_{j'\in S_j}S_{j'}$
to $j$. This is the same with
$J(\phi\circ \varphi)$.

The functoriality of $I \to C(I)$ can be also checked in an elementary 
manner. It is easy to check that the two functors give contra-equivalence.
We omit the detail.
\end{proof}

\begin{definition}
Let $\AS_\csurj$ be a full subcategory of the commutative association schemes of
$\AS_\surj$, and $\Alg_{cHC}$ a full subcategory of $\Alg_{HC}$
consisting of algebras with commutative convolution multiplication.
By the remark at the beginning of Section~\ref{sec:comm}, we have 
a contravariant functor $\AS_\csurj\to \Alg_{cHC}$.
\end{definition}

\begin{corollary}\label{cor:J}
The composite of the three functors $\AS_\csurj \to \Alg_{cHC}$, 
$\Alg_{cHC} \to \Alg_\cs$ forgetting
the Hadamard product, and 
$\Alg_\cs \to \FinSets_\ps$ gives a covariant
functor $AS_\surj \to \FinSets_\ps$. We denote this composition
functor by $J$.
Thus, $(X,R,I) \mapsto J(X,R,I)$ is a functor, which corresponds
$X$ to the sets of primitive idempotents of $A_X$
with respect to (the normalized) convolution products.
\end{corollary}

\begin{definition}
Let $j_0$ denote the element of $J(X,R,I)$, corresponding
to $J_X$, where $J_X$ denotes the matrix in $C(X\times X)$
with components all $1$. This element acts on $C(X)$
as a projection to the space of constant functions. 
We use the symbol $j_0$ for different association schemes,
similarly to $i_0$.
\end{definition}

\subsection{Trivial examples and remarks}
In the theory of commutative association schemes, the set $I$ and 
the set $J$ are considered to be ``dual.'' Thus, it is natural 
to seek the relation between $J$'s for a morphism of association
schemes (while a morphism between $I$'s is given by definition).
The previous result says that there is a natural relation, 
if the morphism of association schemes is surjective. 
We couldn't generalize the result
for general (non-surjective) morphisms.

For a surjective morphism between association schemes, 
the corresponding partial surjection may be a properly partial 
function, and may be not an injection, as follows.
\begin{definition}\label{def:thin}
Let $G$ be a finite group. The triple $(G,R,I)$
where $I=G$ and $R(x,y)=x^{-1}y$ is a (possibly non-commutative)
association scheme called the thin scheme of $G$. 
\end{definition}
\begin{remark} Let $G$ be a finite abelian group, and $f:G \to H$ be 
a surjective homomorphism of abelian groups. 
Then we have the thin scheme of $G$ and the thin scheme of $H$,
and $f$ is a surjective morphism of the association schemes. In this case,
$J$ for $G$ is naturally isomorphic to
the character group $G^\vee$, and one can check that the functor $J$
gives 
$G^\vee \parto H^\vee$, which comes from the natural
inclusion $\iota:H^\vee \to G^\vee$ ($x \in G^\vee$ maps to $y\in H^\vee$
if $\iota(y)=x$, and $x$ is outside the domain if there is no such $y$).
This gives an example that $J(X) \parto J(X')$ is not a map but a partial 
function.
\end{remark}

\begin{remark}
For any association scheme $(X,R,I)$, the mappings $f=\id_X$
and $g:I \to \{0,1\}$, $g(i_0)=0$ and $g(i)=1$ if $i\neq i_0$,
gives a morphism of association schemes
$(X,R,I) \to (X, g\circ R, \{0,1\})$. In this case, $J(X, g\circ R, \{0,1\})$
consists of $j_0$ and $E_X-j_0$. 
If $(X,R,I)$ is commutative, then the image of $j\in J(X,R,I)$ 
in $J(X,g\circ R, \{0,1\})$ is $j_0$ if $j=j_0$, and
$E_X-j_0$ if $j\neq j_0$. This gives an example that
$J(X) \parto J(X')$ is not injective.
\end{remark}

\section{Profinite association schemes}
\subsection{Projective system of association schemes}\label{sec:def-pro}
We recall the notion of projective system.
\begin{definition} (Projective system)
$ $

We work in a fixed category $\cC$.
Let $\Lambda$ be a directed ordered set, namely, $\Lambda$ is non-empty
and 
for any $\alpha,\beta\in \Lambda$, there is 
a $\gamma \in \Lambda$ with $\alpha \leq \gamma$ and $\beta \leq \gamma$.
Let $X_\lambda (\lambda \in \Lambda)$ be a family of objects.
For any $\lambda \leq \lambda'$,
a morphism 
$p_{\lambda',\lambda}:X_{\lambda'} \to X_\lambda$ is specified,
and they satisfy the commutativity condition 
$$
p_{\lambda',\lambda}p_{\lambda'',\lambda'}=p_{\lambda'',\lambda}
$$
for any
$\lambda \leq \lambda' \leq \lambda''$, 
and $p_{\lambda,\lambda=\id_{X_\lambda}}$. 
Then the system $(X_\lambda, p_{\lambda,\lambda'})$ is
called a projective system in $\cC$.
The morphisms $p_{\lambda',\lambda}$
are called the structure morphisms of the projective system.
The dual notion (i.e. the direction of morphism 
is inverted in the definition) is called an inductive system.
\end{definition}
Let $X_\lambda$ be a projective system of association schemes.
Then, each of 
$X_\lambda$, $I_\lambda$ is a projective system of finite sets.
If all morphisms are surjective, then 
by Theorem~\ref{th:convolution}, we have an inductive system 
of Bose-Mesner algebras $A_{X_\lambda}$, and if they are commutative, by Corollary~\ref{cor:J}
a projective system $J_{X_\lambda}:=J(X_\lambda, R_\lambda, I_\lambda)$
(with structure morphisms being partially surjective maps).

\begin{definition}\label{def:proj}
Let $\Lambda$ be a directed ordered set.
A projective system $(X_\lambda, R_\lambda, I_\lambda)$ $(\lambda \in \Lambda)$
in 
$\AS_\surj$ is called a profinite association 
scheme. (This includes the notion of 
usual association scheme, as the case where $\#\Lambda=1$.)
We define profinite sets 
(see Lemma~\ref{lem:profin-top})
$\Xhat:=\varprojlim X_\lambda$
and
$\Ihat:=\varprojlim I_\lambda$
with projective limit topologies. 
A commutative profinite association scheme is 
a projective system of commutative association schemes. In this case, we define $\Jhat$ similarly: 
an element $\hat{j}\in \Jhat$ is 
the set of pair $\lambda \in \Lambda$
and compatible family of elements $j_\mu \in J_\mu$ for all $\mu \geq \lambda$,
divided by the equivalence relation
$(\lambda, j_\mu) \sim (\lambda', j'_{\mu'})$
defined by the existence of $\lambda''\geq \lambda$,
$\lambda''\geq \lambda'$ such that for all $\mu \geq \lambda''$
$j_{\mu}=j'_{\mu}$ holds. An open basis of $\Jhat$ is
given by the inverse image of $j_\lambda \in J_\lambda$
through $\Jhat \parto J_\lambda$.
\end{definition}

\begin{proposition}\label{prop:compact-Hausdorff}
The natural projections
$\Xhat \to X_\lambda$ and $\Ihat \to I_\lambda$ are surjective,
the induced map $\Rhat:\Xhat \times \Xhat \to \Ihat$ is surjective,
and $\Jhat \parto J_\lambda$ is a partial surjection.
\end{proposition}


The first two statements
are well-known for projective systems of compact Hausdorff
spaces with continuous morphisms,
\cite[Proposition~1.1.10, Lemma~1.1.5]{RIBES-ZALESSKII}.
We need to prove the third statement, for partial surjections. 
It follows from the first statement, 
but for the self-containedness and for explaining
the natural (projective limit) topology, we give a proof
for the both.
\begin{lemma} $ $ \label{lem:profin-top}

\begin{enumerate}
\item\label{enum:surj}
Let $X_\lambda$ be a projective system of finite sets and 
surjections. Define $\Xhat$ as $\varprojlim X_\lambda$. 
Then, the natural map $\Xhat \to X_\lambda$ is surjective.
The set $\Xhat$ is called a profinite set, and has a natural topology, called profinite topology,
which is compact Hausdorff. For every $x_\lambda\in X_\lambda$,
its inverse image in $\Xhat$ is clopen, and these form 
an open basis.
\item\label{enum:partial}
Let $J_\lambda$ be a projective system of finite sets and 
partial surjections. Define $\Jhat$ as above. Then, 
the natural partial map $\Jhat \parto J_\lambda$ is a partial surjection.
The set $\Jhat$ has a natural topology, 
which is locally compact and Hausdorff.
For every $j_\lambda\in J_\lambda$,
its inverse image in $\Jhat$ is clopen, and these form 
an open basis.
\item\label{enum:surj-map}
Let $X_\lambda$, $Y_\lambda$ be projective systems of
finite sets over the same directed ordered set $\Lambda$,
and let $f_\lambda:X_\lambda \to Y_\lambda$ be a family of
maps which commute with the structure morphisms. 
This induces a continuous map 
$$
\fhat:\varprojlim X_\lambda \to \varprojlim Y_\lambda.
$$
If every $f_\lambda$ is surjective, then $\fhat$ is surjective.
\end{enumerate}
\end{lemma}
\begin{proof}
(\ref{enum:surj}):
Recall the construction of the projective limit $\varprojlim X_\lambda$
in the category of set. It is a subset of the direct product
$\prod_{\lambda \in \Lambda}X_\lambda$
defined as the intersection of
$$
S_{\mu',\mu}:=
\{(s_\lambda)_\lambda \in \prod_{\lambda \in \Lambda}X_\lambda 
\mid s_\mu=p_{\mu',\mu}(s_{\mu'})
\}
$$
for all $\mu', \mu \in \Lambda$ with $\mu' \geq \mu$. (This means that $(s_\lambda)_\lambda$
belongs to $\varprojlim X_\lambda$ if and only if
they satisfy $p_{\mu',\mu}(s_{\mu'})=s_\mu$ for 
every $\mu'\geq \mu$.) Each finite set is equipped with
the discrete topology, which is compact and Hausdorff. By Tychonoff's
theorem, the product is compact and Hausdorff. Because each 
component is Hausdorff, $S_{\mu',\mu}$ is a closed subset, and
the intersection $\varprojlim X_\lambda$ is compact and Hausdorff. 
Take any $x_\delta \in X_\delta$.
We shall show that this is in the image from $\varprojlim X_\lambda$.
For each $\alpha \in \Lambda$, $\alpha \geq \delta$, 
we consider the closed set
$$
T_\alpha:=
\{(s_\lambda)_\lambda \in \prod_{\lambda \in \Lambda}X_\lambda 
\mid s_\delta=x_\delta, s_\beta=p_{\alpha,\beta}(s_{\alpha})
\mbox{ for all }\beta \leq \alpha
\}.
$$
There is an element $y_\alpha \in X_\alpha$ such that
$p_{\alpha,\delta}(y_\alpha)=x_\delta$, since $p$ is surjective.
Then, there is a system $p_{\alpha,\beta}(y_\alpha)$
which is an element of $T_\alpha$, and thus $T_\alpha$ is nonempty. 
Now we take the intersection of $T_\alpha$ for all $\alpha\geq \delta$.
For any finite number of $\alpha_i$, we may take $\gamma$
as an upper bound of all $\alpha_i$. Then the finite intersection 
contains $T_\gamma$, which is nonempty. By compactness,
the intersection of $T_\alpha$ for $\alpha\geq \delta$
is nonempty. Take an element $(g_\lambda)_\lambda$ 
in the intersection. It lies in $\varprojlim X_\lambda$,
and it projects to $x_\delta$ via 
$\varprojlim X_\lambda \to X_\delta$.
From the definition of the direct product topology, 
the inverse image of $x_\lambda\in X_\lambda$ for
various $\lambda$ forms an open basis. They are 
clopen, since $x\in X_\lambda$ is clopen.

(\ref{enum:partial}):
Choose an element $\star$ which is contained in no $J_\lambda$'s.
Let $J_\lambda^\star$ be $J_\lambda \cup \{\star\}$.
For a partial surjection $p:J_\lambda \parto J_\mu$, 
we associate a surjection $p^\star:J_\lambda^\star \to J_\mu^\star$
as follows. For $x \in \dom(p)$, $p^\star(x)=p(x)$.
For $x \notin \dom(p)$, $p^\star(x)=\star$. In particular, 
$p^\star(\star)=\star$. 
Let us take the projective limit 
as in the previous case to obtain 
$\Jhat^\star:=\varprojlim J_\lambda^\star$. Then for any $\lambda$,
by (\ref{enum:surj}), we have a surjection 
$$p^\star_\lambda:\Jhat^\star \to J_\lambda^\star.$$
It is clear that
there is a unique element 
$\widehat{\star} \in \Jhat^\star$ that is mapped to $\star$
for every $J_\lambda^\star$. 
We define 
$\Jhat':=(\varprojlim J_\lambda^\star) \setminus \{\widehat{\star}\}$,
and a partial surjection $p_\lambda:\Jhat' \parto J_\lambda$ by:
for $x \in \Jhat'$, $x$ is outside of the domain of $p_\lambda$
if $p^\star_\lambda(x)=\star$, and otherwise $p_\lambda(x)=p^\star_\lambda(x)$.
It is easy to check that $p_\lambda$ is a partial surjection
(using the surjectivity of $p^\star_\lambda$),
and $p_\lambda$ for various $\lambda$ is compatible with the
structure morphisms $p_{\mu',\mu}$.

Since $\Jhat^\star$ is compact Hausdorff, its open subset
$\Jhat'$ is locally compact Hausdorff, and has an open basis
as in the case (\ref{enum:surj}).
We need to prove that $\Jhat$ in Definition~\ref{def:proj}
is canonically isomorphic to $\Jhat'$. 
Let $\tilde{J}$ be the set of 
pairs $(\lambda,(j_\mu)_\mu)$ where $(j_\mu)_\mu$ is
a compatible system for $\mu \geq \lambda$.
For such a pair, we assign an element 
of $\Jhat^\star$ as follows. For any $\alpha\in \Lambda$, 
take any $\beta$ with $\beta \geq \alpha$ and $\beta \geq \lambda$.
Define $j_\alpha$ as $p^\star_{\beta,\alpha}(j_\mu)$. This
is well-defined by the commutativity of the structure morphisms,
and gives a map $f:\tilde{J} \to \Jhat'$. This is surjective.
Two pairs $(\lambda,(j_\mu)_\mu)$, $(\lambda',(j'_{\mu'})_{\mu'})$
have the same image if and only if there is a $\lambda''$
with $\lambda''\geq \lambda$, $\lambda''\geq \lambda'$
such that for any $\mu\geq \lambda''$, $j_\mu=j'_\mu$ holds.
This gives a bijection $\Jhat:=\tilde{J}/\sim \to \Jhat'$.

(\ref{enum:surj-map}): See \cite[Lemma~1.1.5]{RIBES-ZALESSKII}.
\end{proof}

\begin{remark}
The above construction gives 
a category equivalence between 
the category of sets with partial maps and
the category of sets with one base point. 
Using this, it is not difficult
to show that $\Jhat$ is the projective limit in the 
category of sets with partial surjections.
\end{remark}
Later it will be proved that, for any commutative profinite association 
scheme, $\Jhat$ has the discrete topology
(Proposition~\ref{prop:discrete}).
\begin{definition}\label{def:Bose-Mesner}
For a profinite association scheme 
$(X_\lambda, R_\lambda, I_\lambda)$, its Bose-Mesner
algebra is defined by 
$A_{\Xhat}:=\varinjlim A_{X_\lambda}$. 
For a commutative profinite association scheme, through the natural isomorphism 
$C(J_\lambda) \to A_{X_\lambda}$,
we have an isomorphism
$\varinjlim C(J_\lambda) \to A_{\Xhat}$
which maps the point-wise product to the convolution product,
and
an isomorphism $\varinjlim C(I_\lambda) \to A_{\Xhat}$
which maps the point-wise product to the Hadamard product.

Let $C(\Xhat \times \Xhat)$
denote the set of complex valued continuous functions 
on $\Xhat \times \Xhat$ (and similarly $C(\Ihat)$ be
the set of continuous functions on $\Ihat$).
Since we have inductive families of injections 
$$
A_{X_\lambda} \to C(X_\lambda \times X_\lambda) \to C(\Xhat \times \Xhat),
$$
by universality, we have a canonical injective 
morphism $A_{\Xhat} \to C(\Xhat\times \Xhat)$.
similarly, we have an injection $\varinjlim C(I_\lambda) \to C(\Ihat)$.
\end{definition}

\begin{definition}
For $j_\lambda \in J_\lambda$, we denote by $E_{j_\lambda} \in A_{X_\lambda}$
the primitive idempotent $j_\lambda$ (the notation introduced to 
distinguish an element of $J$ with an element of $A_{X_\lambda}$).
It is an element of $C(X_\lambda\times X_\lambda)$, and
acts on $C(X_\lambda)$ by the normalized convolution 
defined in Theorem~\ref{th:convolution}.
Let $C(X_\lambda)_{j_\lambda}$ be the image of $E_{j_\lambda}$
by this action.
\end{definition}

\begin{proposition}
Take $j_\lambda \in J_\lambda$. For $\mu \geq \lambda$, 
we may regard $C(X_\lambda)\subset C(X_\mu)$. 
Then $\Psi(E_{j_\lambda})$ 
given in Theorem~\ref{th:convolution}
acts on $g \in C(X_\mu)$
by $\Psi(E_{j_\lambda})\bullet g$. 
The operator $\Psi(E_{j_\lambda})$ gives a splitting 
$C(X_\mu) \to C(X_\lambda)_{j_\lambda}$ to 
$C(X_\lambda)_{j_\lambda}\subset C(X_\mu)$.
\end{proposition}
\begin{proof}
By Theorem~\ref{th:convolution}, 
$\Psi(E_{X_\lambda}) \in \End(C(X_\mu))$ is a projection to $C(X_\lambda)$.
Then 
$$
\Psi(E_{j_\lambda})=\Psi(E_{j_\lambda}\bullet E_{X_\lambda})
=\Psi(E_{j_\lambda})\bullet \Psi(E_{X_\lambda})
$$
is a projection $C(X_\mu) \to C(X_\lambda) \to C(X_\lambda)_{j_\lambda}$.
\end{proof}
\begin{definition}\label{def:isolate}
Let $j$ be an element of $\Jhat$.
Then $j$ is said to be isolated
at $\lambda$ to $j_\lambda$, 
if there is a $j_\lambda \in J_\lambda$
such that the inverse image of $j_\lambda$
is the singleton $\{j\}$.
\end{definition}

\begin{proposition}\label{prop:discrete}
The cardinality of the inverse image of $j_\lambda$
in $J_\mu$ for $\mu \geq \lambda$ does not exceed the dimension of 
$C(X_\lambda)_{j_\lambda}$, which is finite. 
Consequently, for each $j \in \Jhat$, 
there exists a $\lambda$ at which $j$ is isolated.
Thus, $\Jhat$ has the discrete topology.
\end{proposition}
\begin{proof}
Suppose that the inverse image of $j_\lambda$ in $J_\mu$ is $\{ j_1, j_2, \dots ,j_n \}$.
Then the operator $\Psi(E_{j_\lambda}) = \sum_{i=1}^n \Psi(E_{j_i})$ on $C(X_\mu)$
is a projection to $C(X_\lambda)_{j_\lambda}$. 
Thus $C(X_\mu)_{j_i}$ $(i=1,\ldots,n)$ are nonzero direct summands of 
$C(X_\lambda)_{j_\lambda}$, and hence $n$ does not exceed
the dimension of $C(X_\lambda)_{j_\lambda}$. 
This means that for any $\mu \geq \lambda$, 
the cardinality of the inverse image of $j_\lambda$
is bounded, and hence there is a $\mu'$
such that the cardinality is constant for all $\mu \geq \mu'$.
Let us fix $j \in \Jhat$.
Take any $\lambda_0$ such that $j$ is in the domain of the partial map to $J_{\lambda_0}$.
We put $j_{\lambda_0}$ to the image of $j$.
Then by the arguments above, we can find a $\lambda$ 
such that the cardinality of the inverse image of $j_{\lambda_0}$ in $J_{\lambda'}$ is constant for all $\lambda' \geq \lambda$.
We put $j_\lambda$ to the image of $j$ in $J_\lambda$.
Then $j_\lambda$ is in the inverse image of $j_{\lambda_0}$ and 
the cardinality of the inverse image of $j_\lambda$  in $J_{\lambda'}$ should be one for any $\lambda' \geq \lambda$.
This implies that the inverse image of $j_{\lambda}$ in $\Jhat$ is just $\{ j \}$, namely, $j$ is isolated
at $\lambda$.
We have proved that $\{ j \}$ is a clopen subset of $\Jhat$ for any $j \in \Jhat$, 
and hence $\Jhat$ has the discrete topology.
\end{proof}
\begin{corollary}\label{cor:E-j}
We have $\varinjlim C(J_\lambda)=\bigoplus_{j\in \Jhat} \C\cdot j$.
We write $E_j$ for the image of $j\in \Jhat$ in $A_\Xhat$.
Then $\bigoplus_{j\in \Jhat} \C\cdot E_j = A_\Xhat$.
We may use the notation $C_c(\Jhat)$ for 
$\varinjlim C(J_\lambda)$, since it is the set of
functions (automatically continuous) on $\Jhat$
with finite ($=$compact)
support. Through the isomorphism $C_c(\Jhat) \to A_\Xhat$, $E_j\in A_\Xhat$ corresponds to the indicator function 
$\delta_j \in C_c(\Jhat)$.
\end{corollary}

\begin{remark}
In the Pontryagin duality, the dual of a compact abelian group
is a discrete group. Since $\Ihat$ is profinite and compact,
the above discreteness of $\Jhat$ is an analogue to this fact. 
\end{remark}

We summarize the results of this subsection, to compare with
the case of finite commutative association schemes.
\begin{theorem}\label{th:commutative}
Let $X_\lambda$ be a commutative profinite association scheme. 
Notations are as in Definition~\ref{def:Bose-Mesner} and Corollary~\ref{cor:E-j}.
We have the following commutative diagram, where 
all the vertical arrows are injections:
\begin{equation}\label{eq:commutative}
\begin{array}{ccccccccc}
& & C(\Jhat) & &  C(\Xhat \times \Xhat) & \leftarrow & C(\Ihat) & &\\
& &\uparrow & & \uparrow & & \uparrow & &\\
C_c(\Jhat)&=&\varinjlim C(J_\lambda) &\simeq& A_{\Xhat}
  &\simeq& \varinjlim C(I_\lambda) &=& C_{lc}(\Ihat), \\
\end{array}
\end{equation}
where:
\begin{enumerate}
 \item The set $\Jhat$ is discrete.
  The image of the left vertical arrow is the 
  set of compact support functions $C_c(\Jhat)$
  (i.e. takes 0 except but finite points
  in $\Jhat$). 
 \item The image of the right vertical arrow is the
 set of locally constant functions $C_{lc}(\Ihat)$, 
 which is dense with respect to the supremum norm, see the next lemma.
\end{enumerate}
\end{theorem}

\begin{lemma}\label{lem:density}
Let $\Xhat$ be the projective limit of a projective system of finite sets $X_\lambda$
for a directed ordered set $\Lambda$. 
\begin{enumerate}
\item \label{enum:separate}
Let $C$ be a compact subspace of $\Xhat$, and $O_a$ $(a\in A)$
be an open covering of $C$. 
Then, there exists a $\mu \in \Lambda$ such that 
by putting $S$ to the image of $C$ in $X_\mu$, 
the inverse image of $s\in S$ in $\Xhat$
constitutes a (finite and disjoint) open covering of $C$, and each of
these open sets is contained in one of $O_a$.
\item \label{enum:dense}
The image of $\varinjlim C(X_\lambda) \to C(\Xhat)$ is 
the set of locally constant functions on $\Xhat$, and is dense
with respect to the supremum norm.
\end{enumerate}
\end{lemma}
\begin{proof}

(\ref{enum:separate}): For $x_\lambda \in X_\lambda$, 
we denote by $B_{x_\lambda}$ the inverse image of $x_\lambda$
in $\Xhat$, which we call a $\lambda$-ball. 
This gives an open basis, and thus
every $O_a$ is a union of $\lambda$-balls for various $\lambda$.
Since $C$ is compact and covered by these balls, 
we may find a finite number of these balls which covers $C$
and each contained in some $O_a$. We take an upper bound $\mu$
for these finite number of $\lambda$. Let $S$
be the image of $C$ in $X_\mu$. Then, $B_{s_\mu} (s_\mu \in S)$
are finite, clopen, and mutually disjoint balls covering $C$,
such that each ball is contained in some $O_a$.

(\ref{enum:dense})
Let $f$ be a locally constant function. Then, for any $x\in \Xhat$,
there is an open neighborhood $O_x$ so that $f$ is constant on $O_x$.
These constitute an open covering of $\Xhat$.
Applying (\ref{enum:separate}) for $C=\Xhat$ to find $\mu$ with $S=X_\mu$,
where each $\mu$-ball is contained in some $O_x$, and hence $f$
is constant on each ball. This means that $f$ is an image 
of an element in $C(X_\mu)$.
Conversely, any function 
in $C(X_\mu)$ clearly maps to a locally constant function on $\Xhat$.

We shall show the density. For any $f\in C(\Xhat)$
and $\epsilon>0$, by continuity, for all $x\in \Xhat$ we have
an open neighborhood $O_x$
such that $y\in O_x$
implies $|f(x)-f(y)|<\epsilon$. By applying (\ref{enum:separate})
for $C=\Xhat$ to find a $\mu$ with $S=X_\mu$.
We see that the $\mu$-balls cover $\Xhat$, and for each $x_\mu \in X_\mu$,
we define $g(x_\mu)=f(x')$ by choosing any $x'$ in $B_{x_\mu}$.
Denote the projection by $\pr:\Xhat \to X_\mu$.
Then for any $y\in \Xhat$, $g(\pr(y))=f(x')$ for some 
$x' \in B_{\pr(y)}$. Since $B_{\pr(y)}\subset O_x$ for some $x$,
the values of $f$ in this ball differ at most by $2\epsilon$,
namely, 
$$
|g(\pr(y))-f(y)|=|f(x')-f(y)|<2\epsilon.
$$
This shows that
$$
||\pr^\dagger(g)-f||_{\sup}<2\epsilon,
$$
which is the desired density.
\end{proof}

\subsection{Inner product and orthogonality}
\begin{proposition}
For a profinite association scheme, $\Xhat$ has a natural 
probability regular Borel measure $\mu$, whose push-forward is the normalized 
(probability) counting measure on $X_\lambda$ for each $\lambda$.
If $\Xhat$ is a profinite group $G$ considered
as a projective system of $G/N_\lambda$ for (finite index) open normal 
subgroups, then the above $\mu$ coincides with the probability Haar measure.
\end{proposition}
\begin{proof} 
This is a direct consequence of Choksi~\cite{CHOKSI}
and a general property of measures (see Halmos~\cite{HALMOS}),
as follows.
Each $X_\lambda$ is finite discrete (hence 
compact Hausdorff) with 
normalized counting measure,
and the measure is preserved by 
the uniformness of 
the cardinality of each fiber
as in Proposition~\ref{prop:fiber}.
This means that the system is an inverse family of
measured spaces \cite[Definition~2]{CHOKSI}, and hence
$\Xhat$ has a weak measure structure: a set $M$ of ring of sets
and a function $\mu$ from $M$ to non-negative
real numbers, such that $\Xhat \to X_\lambda$
preserves the measure. In this case, $M$ is the set of open sets,
and $\mu(\Xhat)=1$.
Then \cite[Theorem~2.2]{CHOKSI} tells that
this $\mu$ on $\Xhat$ is extended to a measure on 
$\sigma$-ring generated by $M$,
namely, the Borel sets. There it is also proved
that for any measurable set $E$ with finite measure and any $\epsilon>0$, 
there is a compact set $C \subset E$ with $\mu(E)<\mu(C)+\epsilon$.
This proves the inner regularity of $\mu$ (\cite[\S52, P.224]{HALMOS}).
By taking the complement, for any $\epsilon>0$ and measurable set 
$E^c$, there is an open $U \supset E^c$ with $1-\mu(E^c)<1-\mu(U)+\epsilon$,
i.e., $\mu(E^c)>\mu(U)-\epsilon$, which proves the outer regularity of $\mu$,
and hence $\mu$ is regular.

Let $\nu$ be the probability Haar measure of $G$. Then, for 
any (finite index) open normal subgroup $N$, $\nu(N)$ is $\frac{1}{\#(G/N)}$,
since $\#(G/N)$ cosets of $N$ disjointly covers $G$. This gives
an inverse family of measured spaces, and the above construction 
gives a regular Borel measure $\mu$. The Haar measure is regular 
\cite[\S58, Theorem~B]{HALMOS}. For the coset $gN$, $\mu(gN)=\nu(gN)$.
The cosets $gN$ for $g$ and $N$ varying make an open basis.
By the outer regularity, for any compact set $E$ and any $\epsilon>0$, 
there is an open set $O\supset E$ with $\mu(E)>\mu(O)-\epsilon$.
Use Lemma~\ref{lem:density} for $C=E$ and the open covering $O$,
to find $\mu$, $S\subset X_\mu$. Then $B_{s_\mu} (s_\mu\in S)$
constitutes an open cover $\cC$ of $E$ consisting of a finite number of
disjoint cosets (contained in $O$) such that the sum of 
the measure of each coset in $\cC$, denoted by $\mu(\cC)$
$(\mu(\cC)\leq \mu(O))$, satisfying
\begin{equation}\label{eq:epsilon}
\mu(\cC)\geq \mu(E)\geq \mu(\cC)-\epsilon, 
\end{equation}
but if the same argument simultaneously applied for both $\mu$
and $\nu$ for the $O$ satisfying both
$\mu(E)>\mu(O)-\epsilon$ and
$\nu(E)>\nu(O)-\epsilon$, we may take the same $\cC$ for the both and 
the same inequality (\ref{eq:epsilon})
holds for $\nu$, and because consisting of 
a finite number of 
disjoint cosets, $\nu(\cC)=\mu(\cC)$ holds,
which implies $\mu(E)=\nu(E)$. This means that the restrictions of
$\nu$ and $\mu$ to compact sets give the same regular content,
which implies $\nu=\mu$ by \cite[\S54, Theorem~B]{HALMOS}.
\end{proof}

\begin{proposition}\label{def:inner-product}
For a finite set $X$, we define a (normalized) Hermitian inner product 
on $C(X)$ by 
$$
 (f,g):=\frac{1}{\#X}\sum_{x\in X}f(x)\overline{g(x)}.
$$
For a profinite association scheme, it is defined on $C(\Xhat)$
by 
$$
(f,g):=\int_{x\in \Xhat}f(x)\overline{g(x)} d\mu(x).
$$
Then $C(X_\lambda)\to C(\Xhat)$ preserves the inner products.
\end{proposition}
This is clear since $\Xhat \to X_\lambda$ preserves the measures.
\begin{definition}
Let $V,W$ be $\C$-vector spaces with Hermitian inner products.
For a linear morphism $A:V \to W$, its conjugate $A^*:W \to V$ 
(which may not exists)
is characterized by the property
$$
(Af, g)_W=(f,A^*g)_V
$$ 
for any $f\in V$, $g \in W$. It is unique if exists.
An endomorphism $A:W \to W$ is said to be Hermitian if $A=A^*$, as usual.

Let $W$ be again a $\C$-vector space with Hermitian inner product,
and $V$ is a subspace with restricted inner product.
Let $V^\perp$ be the space of vectors which is orthogonal to 
any vector in $V$. If $V\oplus V^\perp$ is equal to $W$, 
then $V$ is said to be an orthogonal component of $W$. 
\end{definition}
\begin{definition}\label{def:ortho-proj}
Let $W$ be a $\C$-linear space with Hermitian inner product.
For a linear subspace $V\subset W$, 
let $p:W \to W$ be a projection to $V$, i.e., the image of $p$
is $V$ and $p|_V=\id_V$. Then by linear algebra $W$ is decomposed 
as the direct sum, $W=V \oplus \Ker p$.
We say that $p$ is an orthogonal projector to $V$ if
$V \oplus \Ker p$ is an orthogonal direct sum.
An orthogonal projector to $V$ exists uniquely, if $V$ is an orthogonal component.
\end{definition}

\begin{proposition}\label{prop:abuse}
Let $W$ be a $\C$-linear space with Hermitian inner product,
$V\subset W$ be an orthogonal component with 
orthogonal projector $p$,
and $\iota:V \hookrightarrow W$ be the 
inclusion.
We denote by $q: W \to V$ the restriction of the codomain of $p$.
Then we have $\iota^*=q$ and $q^*=\iota$.
For $A:V \to V$, if $A^*$ exists, then
\begin{equation}\label{eq:subconj}
(\iota \circ A \circ q)^*=\iota \circ A^* \circ q:V \to V.
\end{equation}
We sometimes denote $\iota \circ A \circ q$ simply by $A$. We make a remark
when this abuse of notation is used.
\end{proposition}
\begin{proof}
For $w\in W$ and $v\in V$, write $w=w_V+w_V'$ with $w_V\in V$ and $w_V'\in V^\perp$.
Then
$$
(qw,v)_V=(w_V,v)_V=(\iota w_V,\iota v)_W=(w_V+w_V',\iota v)_W=(w,\iota v)_W,
$$
which shows that $\iota^*=q$ and $q^*=\iota$.
Thus (\ref{eq:subconj}) follows.
\end{proof}

\begin{proposition}\label{prop:hermite-projection}
Let $W$ be a $\C$-vector space with Hermitian inner product, and
$p:W \to W$ an idempotent, namely, $p^2=p$.
By linear algebra, $W=\Img(p)\oplus\Ker(p)$
and $p$ is a projector to $V:=\Img(p)$. Then, $p$ is an orthogonal
projector if and only if $p=p^*$.
If $A \in \End(V)$ is an orthogonal projector to $U\subset V$, 
then $\iota A q : W \rightarrow W$ is a projector to $U\subset W$.
\end{proposition}
\begin{proof}
If $p^2=p$, then $W=V\oplus \Ker p$ for $V=\Img(p)$.
Thus any element $w,w'\in W$ is written as $v+k, v'+k'$, respectively.
If $V \oplus \Ker p$ is orthogonal, 
$$
(p(v+k), v'+k')=(pv,v'+k')=(v,v')=(v+k,pv')=(v+k,p(v'+k')),
$$
which shows $p=p^*$. For the converse, assume $p=p^*$. 
Then for $v\in V$, $k\in \Ker p$, $(v,k)=(pv,k)=(v,pk)=(v,0)=0$,
and hence $W$ and $\Ker p$ are mutually orthogonal.
For the last, (\ref{eq:subconj}) implies that 
$\iota A q$ is an Hermitian operator, and by $q\circ \iota=\id_V$
is an idempotent, whose image is $U$, hence is an orthogonal
projector to $U$.
\end{proof}

\begin{definition}\label{def:meas_on_fiber}
Let $(X_\lambda)$ be a profinite association scheme, fix $\lambda\in \Lambda$ and $x_\lambda \in X_\lambda$.
For the surjection $\pr:\Xhat \to X_\lambda$, 
the fiber $\pr^{-1}(x_\lambda)$ at $x_\lambda$ 
is a clopen subset of $\Xhat$.
Thus, we have the restriction of the measure $\mu$ on $\Xhat$
to $\pr^{-1}(x_\lambda)$, denoted by the same symbol $(\pr^{-1}(x_\lambda), \mu)$.
Note that the volume of the fiber $\pr^{-1}(x_\lambda)$ is $1/\#X_{\lambda}$.
\end{definition}
Since $\Xhat = \coprod_{x_\lambda \in X_\lambda}\pr^{-1}(x_\lambda)$, the following proposition holds.
\begin{proposition}\label{prop:FubiniXhat}
For any $\lambda$ and $f \in C(\Xhat)$, we have
\[
\int_{x \in \Xhat} f(x) d\mu(x) = \sum_{x_\lambda \in X_\lambda} \int_{y \in \pr^{-1}(x_\lambda)} f(y) d\mu(y).
\]
\end{proposition}

\begin{proposition}\label{prop:q}
Let $(X_\lambda)$ be a profinite association scheme.
Then, by the surjection $\pr:\Xhat \to X_\lambda$,
$C(X_\lambda)$ is identified with a subspace
of $C(\Xhat)$ with the inner product given by restriction. 
There is an orthogonal projection $q:C(\Xhat) \to C(X_\lambda)$
given by averaging over each fiber: for $f \in C(\Xhat)$ define
$$
q(f)(x_\lambda)=
 \frac{1}{\mu(\pr^{-1}(x_\lambda))}
 \int_{x\in \pr^{-1}(x_\lambda)}f(x) d\mu(x),
$$
where $\mu(\pr^{-1}(x_\lambda))=\frac{1}{\#X_\lambda}$.
Thus, $C(X_\lambda)$ is an orthogonal component of $C(\Xhat)$.
\end{proposition}
\begin{proof}
The preservation of the inner products is shown in Proposition~\ref{def:inner-product}.
For any $g \in C(X_\lambda)$ its image is $g\circ \pr \in C(\Xhat)$ and we have
\begin{eqnarray*}
q(g\circ \pr)(x_\lambda)&=&\frac{1}{\mu(\pr^{-1}(x_\lambda))}\int_{y\in \pr^{-1}(x_\lambda)}g(\pr(y))d\mu(y) \\
&=&g(x_\lambda)\frac{1}{\mu(\pr^{-1}(x_\lambda))}\int_{y\in \pr^{-1}(x_\lambda)}1\cdot d\mu(y)\\
&=&g(x_\lambda).
\end{eqnarray*}
This implies that $q$ is a projection to $C(X_\lambda)$.

To show that $q$ is an orthogonal projection, we take $f\in \Ker q$.
This means that the average of $f$ on 
each fiber of $\pr$ is zero. To show that $q$ is orthogonal, 
it suffices to show that $(f,g\circ \pr)_\Xhat=0$ for $g\in C(X_\lambda)$
by Definition~\ref{def:ortho-proj}, 
but by applying Proposition \ref{prop:FubiniXhat}, we have 
\begin{eqnarray*}
(f,g\circ \pr)_\Xhat&=&\int_{x \in \Xhat} f(x) \overline{g(\pr(x))} d\mu(x) \\
&=&
\sum_{x'\in X_\lambda}
\int_{x\in \pr^{-1}(x')}f(x)\overline{g(\pr(x))} d\mu(x) \\
&=&
\sum_{x'\in X_\lambda}
\overline{g(x')}\int_{x\in \pr^{-1}(x')}f(x) d\mu(x) \\
&=&
\sum_{x'\in X_\lambda}
\overline{g(x')}\cdot 0 = 0. \\
\end{eqnarray*}
\end{proof}

\begin{proposition}\label{prop:orthogonal}$ $

\begin{enumerate} 
    \item \label{enum:conv} The (convolution) algebra $A_\Xhat$ acts on $C(\Xhat)$,
which makes $C(\Xhat)$ an $A_\Xhat$-module. This is compatible
with the action of $A_{X_\lambda}$ on $C(X_\lambda)$, in the sense of Corollary~\ref{cor:mod}.
    \item The convolution unit $E_{X_\lambda} \in A_{X_\lambda} \subset A_{\Xhat}$
acts as an orthogonal projector to $C(X_\lambda)$.
    \item Let $j \in \Jhat$. Take a $\lambda$ 
    where $j$ is isolated to $j_\lambda$,
    and consider $E_{j_\lambda}\in A_{X_\lambda}$.
    By the abuse of 
    notation in Proposition~\ref{prop:abuse},
    $E_{j_\lambda}$ extends to the orthogonal projector
    $$
    E_j:C(\Xhat) \to C(\Xhat) 
    $$
    that projects to the domain 
    of the composition
    $$
    C(X_{\lambda})_{j_\lambda} \to C(X_{\lambda}) \to C(\Xhat).
    $$
    We define 
    $$
    C(\Xhat)_j:=C(X_\lambda)_{j_\lambda}.
    $$
    This is independent of the choice of the $\lambda$.
    \item For $j \neq j'$, 
$E_j\bullet E_{j'}=0$.
    \item 
We have an orthogonal decomposition
$$
C(X_\lambda)=\bigoplus_{j\in J'} C(\Xhat)_j
$$
for the domain $J'$ of the partial map from $\Jhat$ to $J_\lambda$.
Note that $J'$ is finite.
\end{enumerate}
\end{proposition}
\begin{proof}
The action is given by: $A\in A_{X_\lambda}$
acts on $f\in C(\Xhat)$ by
$$
(A \bullet f)(x):=\int_{y\in \Xhat}A(\pr_\lambda(x),\pr_\lambda(y))f(y)d\mu(y).
$$
By Proposition~\ref{prop:FubiniXhat}, this is the same with 
\begin{equation}\label{eq:actionA}
\sum_{y'\in X_\lambda}A(\pr_\lambda(x),y')
(\int_{y\in \pr_\lambda^{-1}(y')}f(y)d\mu(y)).
\end{equation}
Comparing with $q$ in Proposition~\ref{prop:q},
we have
\begin{eqnarray*}
(A \bullet_{\Xhat} f)(x)
&=&
\sum_{y'\in X_\lambda}A(\pr_\lambda(x),y')
(\int_{y\in \pr_\lambda^{-1}(y')}f(y)d\mu(y)) \\
&=&
\sum_{y'\in X_\lambda}A(\pr_\lambda(x),y')
\frac{1}{\#X_\lambda} q(f)(y') \\
&=&
(A\bullet_{X_\lambda} q(f))(\pr_\lambda(x)), 
\end{eqnarray*}
which shows that
\begin{equation}\label{eq:iAq}
A= \iota A q,
\end{equation}
where $\iota:C(X_\lambda)\to C(\Xhat)$ is the injection $\pr_\lambda^\dagger$
and $A$ in the right hand side is an endomorphism of $C(X_\lambda)$.
This and $q\circ \iota=\id_{C(X_\lambda)}$ show that $C(\Xhat)$ is an $A_\Xhat$-module.
The well-definedness 
(independence of the choice of $\lambda$)
follows by taking a sufficiently large $\mu$.

By (\ref{eq:iAq}) and $E_{X_\lambda}=\id_{C(X_\lambda)}$ on $C(X_\lambda)$ 
show that 
$$
E_{X_\lambda}=\iota \circ q,
$$
which is the orthogonal projection to $C(X_\lambda)$.

For $j\in J$, take a $\lambda$ where $j$ is isolated to $j_\lambda \in J_\lambda$.
By (\ref{eq:iAq}) we have
\[
E_{j_\lambda}=\iota E_{j_\lambda} q,
\]
where $E_{j_\lambda}$ in the right hand side is an
endomorphism of $C(X_\lambda)$.
Hence to prove that $E_{j_\lambda}$ 
is an orthogonal projector 
$C(\Xhat) \to C(X_\lambda)_{j_\lambda}$,
it suffices to show that 
$E_{j_\lambda} \in A_{X_\lambda}$ is an orthogonal projector 
$C(X_\lambda) \to C(X_\lambda)_{j_\lambda}$ by Proposition~\ref{prop:hermite-projection}.
The orthogonality of $E_{j_\lambda}\in A_{X_\lambda}$ is well known \cite{Bannai}. 
(Since $A_{X_\lambda}$ is closed under the anti-isomorphism $*$ of ring, $E_{j_\lambda}^*$
must be a primitive idempotent, and since $E_{j_\lambda}E_{j_\lambda}^*\neq 0$,
they must coincide, hence $E_{j_\lambda}$ is Hermitian and an orthogonal projection 
by Proposition~\ref{prop:hermite-projection}.)
We define $E_j:=E_{j_\lambda}$, where
$$
E_j :C(\Xhat) \to C(\Xhat)_j:=C(X_\lambda)_{j_\lambda}.
$$
For the well-definedness, we take $\lambda$, $\lambda'$
where $j$ is isolated. There exists a $\mu$
with $\mu\geq \lambda$, $\mu\geq \lambda'$. Then $j$
is isolated to $j_\mu$ at $\mu$, and by (\ref{eq:iAq}) it is not difficult
to see that the $E_j$ defined using $\mu$
is the same as those defined using $\lambda,\lambda'$.
For $E_j\bullet E_{j'}=0$,
take a $\lambda$ where both $j$, $j'$ 
are isolated to $j_\lambda$, $j'_\lambda$, respectively.
Then, $j_\lambda \neq j_\lambda'$,
$E_{j_\lambda}, E_{j'_\lambda} \in A_{X_\lambda}$, 
and $E_{j_\lambda}\bullet E_{j'_\lambda}=0$. Thus, the relation 
holds in the inductive limit $A_{\Xhat}$.

For fixed $\lambda$, it is well-known \cite{Bannai} that
we have an orthogonal decomposition
$$
C(X_\lambda)=\bigoplus_{j_\lambda \in J_\lambda} C(X_\lambda)_{j_\lambda}.
$$
(This follows from the orthogonality of $E_{j_\lambda}$.)
Since each $C(X_\lambda)_{j_\lambda}$ is an orthogonal
direct sum of $C(\Xhat)_j$ over the finite number of $j\in \Jhat$ 
whose image is $j_\lambda$, the claim follows.
\end{proof}
The following is an analogue to Peter-Weyl Theorem.
\begin{theorem}
We have an orthogonal direct decomposition
$$
\varinjlim C(X_\lambda)=\bigoplus_{j\in J}C(\Xhat)_j.
$$
Hence the right hand side is dense in $C(\Xhat)$
with respect to the supremum norm.
\end{theorem}
\begin{proof}
Proposition~\ref{prop:orthogonal} implies
that $C(\Xhat)_j$ are mutually orthogonal, and 
their orthogonal sum is a subspace of $C(\Xhat)$.
For any $\lambda$, $C(X_\lambda)$ is an orthogonal 
sum of a finite number of $C(\Xhat)_j$. Thus 
$\varinjlim C(X_\lambda)$ is the direct sum
of $C(\Xhat)_j$ for $j\in \Jhat$.
The density follows from
Lemma~\ref{lem:density}.
\end{proof}
\subsection{Measures and Fourier analysis}
Since $\Xhat$ has a natural probability measure,
$\Ihat$ has a pushforward measure via $\Xhat \times \Xhat \to \Ihat$.
It will be shown that $\Jhat$ has a natural measure,
and these measures are
closely related to the classical notions 
of the multiplicity $m_{j_\lambda}$ for $J_\lambda$, 
and the valency $k_{i_\lambda}$ for $I_\lambda$ (see \cite{Bannai}\cite{DELSARTE}).
\begin{definition}
Let $\mu_\Ihat$ be the pushforward measure on $\Ihat$ 
from $\Xhat \times \Xhat$. This is characterized as follows.
For $i_\lambda \in I_\lambda$, let $\pr_{I_\lambda}^{-1}(i)\subset \Ihat$
be the open set. Then 
$$
\mu_\Ihat(\pr_{I_\lambda}^{-1}(i_\lambda))
=
\mu_{\Xhat\times\Xhat}
(R_\lambda \circ \pr_{X_\lambda\times X_\lambda})^{-1}(i_\lambda).
$$
\end{definition}
\begin{proposition}\label{prop:k-i}
In the above definition, 
$$
\mu_\Ihat(\pr_{I_\lambda}^{-1}(i_\lambda))=\frac{k_{i_\lambda}}{\#X_\lambda}
$$
holds, where $k_{i_\lambda}$ is the valency of the 
adjacency matrix $A_{i_\lambda} \in A_{X_\lambda}$.
\end{proposition}
\begin{proof}
Recall that $A_{i_\lambda}$ is the
indicator function of $R_\lambda^{-1}(i_\lambda)$,
and the sum of the rows is an integer $k_{i_\lambda}$ called the valency
\cite{Bannai}.
Thus, the volume $\mu_{\Ihat}(\pr_{I_\lambda}^{-1}(i_\lambda))$
is the volume of 
$\mu_{\Xhat}(\pr_{X_\lambda\times X_\lambda})^{-1}R_\lambda^{-1}(i_\lambda)$,
which is
\begin{eqnarray*}
& & \\
\int_{x,y\in \Xhat\times \Xhat} 
A_{i_\lambda}(\pr_{X_\lambda}(x),\pr_{X_\lambda}(y)) 
d\mu(x)d\mu(y)
&=&
\frac{1}{\#X_\lambda^2}\sum_{x',y'\in X_\lambda \times X_\lambda}
A_{i_\lambda}(x',y') \\
&=& 
\frac{k_{i_\lambda}}{\#X_\lambda}.
\end{eqnarray*}
\end{proof}

\begin{definition}
The algebra $A_{X_\lambda}$ is closed under the adjoint ${}^*$.
We define the Hilbert-Schmidt inner product,
which is Hermitian, as follows. For $A,B\in A_{X_\lambda}$,
$$
(A,B)_\HS:= \sum_{x \in X_\lambda} (A\bullet B^*)(x,x).
$$
For $A, B \in A_\Xhat$, we define its Hilbert-Schmidt inner product 
by 
$$
(A,B)_\HS:= \int_{x \in \Xhat} (A\bullet B^*)(x,x) d\mu(x).
$$
Then inclusion $A_{X_\lambda} \to A_{\Xhat}$ preserves
the inner products.
\end{definition}
\begin{proof}
For $B \in A_{X_\lambda}$,
$B^* \in A_{X_\lambda}$ follows because of the definition of
association schemes, and the Hilbert-Schmidt inner product
is defined. 
For $B \in A_{\Xhat}$, we must show that $B^*$
exists in $A_{\Xhat}$ and the integration converges. However, 
$B \in A_{X_\lambda}$ for some $\lambda$, and 
$B^* \in A_{X_\lambda}$.
Since $C(X_\lambda)$ is 
an orthogonal component of $C(\Xhat)$, 
$B^*=\iota B^* q$ is
the adjoint as an operator on $C(\Xhat)$,
by the abuse of notation in Proposition~\ref{prop:abuse}.
The notation $(A\bullet B^*)(x,x)$ comes from the middle vertical injection
in (\ref{eq:commutative}).
Since
$A, B^* \in A_{X_\lambda} \subset C(X_\lambda \times X_\lambda)$,
$$
\int_{x \in \Xhat} (A\bullet B^*)(x,x) d\mu(x) = 
\sum_{x' \in X_\lambda} (A\bullet B^*)(x',x'),
$$
which is finite, satisfying the axioms of Hermitian
inner product, and compatible with $A_{X_\lambda} \to A_{\Xhat}$.
\end{proof}
\begin{proposition}\label{prop:HS-I}
Through the natural isomorphism 
$$
\varinjlim C(I_\lambda) \to A_{\Xhat},
$$
we obtain Hermitian inner product on $\varinjlim C(I_\lambda)$,
which is denoted by $(-,-)_\HS$.
Then for $f,g\in \varinjlim C(I_\lambda)$,
we have
$$
(f,g)_\HS=\int_{i\in \Ihat}f(i)\overline{g(i)} d\mu_\Ihat (i).
$$
\end{proposition}
\begin{proof}
We may assume $f,g \in C(I_\lambda)$ for some $\lambda$.
Then by bilinearity it suffices to check the equality
for $f=\delta_{i_\lambda}$, $g=\delta_{i_\lambda'}$.
Then, the right hand side is $\delta_{i_\lambda,i_\lambda'}k_{i_\lambda}/\#X_{\lambda}$
by Proposition~\ref{prop:k-i}.
In $A_\Xhat$, these functions correspond to 
$f\mapsto A_{i_\lambda}, g\mapsto A_{i_\lambda'}$, and their inner product is
\begin{eqnarray*}
& & 
\int_{x,y\in \Xhat\times \Xhat} 
A_{i_\lambda}(\pr_{X_\lambda}(x),\pr_{X_\lambda}(y)) A_{i_\lambda'}(\pr_{X_\lambda}(y),\pr_{X_\lambda}(x))^* 
d\mu(x)d\mu(y)\\
&=&
\frac{1}{\#X_\lambda^2}\sum_{x',y'\in X_\lambda \times X_\lambda}
A_{i_\lambda}(x',y') A_{i_\lambda'}(x',y')\\
&=& 
\delta_{i_\lambda,i_\lambda'}\frac{k_{i_\lambda}}{\#X_\lambda}.
\end{eqnarray*}
This uses a well-known orthogonality of $A_{i_\lambda}$:
for any $x\in X_\lambda$ and $i_1,i_2\in I_\lambda$, 
the $(x,x)$-component of the matrix product $A_{i_1}\cdot^{t}A_{i_2}$
is the number of $y\in X_\lambda$ with $R(x,y)=i_1$ and $R(x,y)=i_2$,
which is zero unless $i_1=i_2$, and in the equal case $k_{i_1}$.
\end{proof}
\begin{lemma}
For a commutative association scheme $(X,R,I)$, we defined $J$
as the set of primitive idempotents $E_j$
with respect to the product $\bullet$ in Theorem~\ref{th:convolution}.
We put a positive measure $m_j$ for $j\in J$, so that
the inner product with respect to $m_j$ in $C(J)$ is
isometric to that of $A_X$ with respect to $\HS$ inner product. 
Then, $m_j$ is the multiplicity
of $j$, defined as $\dim(C(X)_j)$
\cite{Bannai}.
\end{lemma}
\begin{proof}
Let $E_j, E_{j'}\in A_{X}$ be two primitive
idempotents. The computation
\begin{eqnarray*}
(E_j,E_{j'})_{\HS}
=\sum_{x\in X}E_j\bullet E_{j'}^*(x,x)
=\delta_{jj'}\tr(E_j)=\delta_{jj'}m_j
\end{eqnarray*}
implies that if we define $\mu_{J}(j):=m_j$,
then the corresponding inner product
is compatible with $C(J) \to A_X$ with 
respect to the $\HS$-inner product on $A$.
\end{proof}

\begin{proposition}\label{prop:HS-J}
For a projective association scheme $(X_\lambda,R_\lambda,I_\lambda)$, 
we define a measure $\mu_\Jhat$
on the discrete set $\Jhat$ by 
$\mu_\Jhat(j)=\dim C(\Xhat)_j$. 
Then, the inner product on $\varinjlim C(J_\lambda)$
defined by 
$$
(f,g)_\HS=\sum_{j \in \Jhat} f(j)\overline{g(j)}\mu_\Jhat(j)
$$
makes $\varinjlim C(J_\lambda) \cong A_\Xhat$ isometric with respect
to $\HS$ inner product.
\end{proposition}
\begin{proof}
Any $f,g\in \varinjlim C(J_\lambda)$ are 
contained in some $C(J_\lambda)$, where
any $j$ in the support of $f$ or $g$ is isolated at $\lambda$.
The functions $f,g$ are linear combinations of orthogonal basis $E_{j}$
of $C(J_\lambda)$,
namely, $f=\sum_{j}f(j)E_j$ and $g=\sum_{j}g(j)E_j$ in $A_{X_\lambda}$.
The previous lemma says that
$(E_{j},E_{j'})_\HS=\delta_{jj'}\dim C(\Xhat)_j$, 
which proves the proposition.
\end{proof}
In sum, we have the following form of Fourier transform.
\begin{theorem}
The isomorphism of vector spaces
$\varinjlim C(J_\lambda) \cong A_{\Xhat} \cong \varinjlim C(I_\lambda)$
is an isometry between
the space of compact support functions $C_c(\Jhat)$ with inner product
defined in Proposition~\ref{prop:HS-J}
and the space of locally constant functions $C_{lc}(\Ihat)$
with inner product defined in Proposition~\ref{prop:HS-I},
which maps the componentwise product in $C_c(\Jhat)$
to the convolution product in $C_{lc}(\Ihat)$. 
\end{theorem}

\begin{remark}
For a finite association scheme $(X,R,I)$, both $E_j$ $(j\in J)$
and $A_i$ $(i \in I)$ are orthogonal bases of $A_X$ 
with respect to $\HS$-inner product. The isometry
$$
C(I) \to C(J)
$$
is given by 
\begin{eqnarray*}
A_i& \mapsto & \sum_{j \in J} p_i(j)E_j.  \\
\end{eqnarray*}
The isometry shows that
$$
(A_i,A_{i'})_\HS=\delta_{ii'}\mu_I(i)=\delta_{ii'}\frac{k_i}{\#X}
$$ 
is equal to 
$$
\sum_{j\in J}p_i(j)\overline{p_{i'}(j)}m_j,
$$
which is one of the two orthogonalities of eigenmatrices 
as classically known \cite{Bannai}.
(Here the factor $\#X$ appears in a different manner because
of the normalization in Theorem~\ref{th:convolution}.
Our $E_j$ is $\#X$ times the classical primitive idempotent,
which makes $p_i(j)=\frac{1}{\#X}P_i(j)$ where $P_i(j)$ are the
components of the standard eigen matrices.)
\end{remark}
The next is an analogue of the Fourier transform for $L^2$ spaces.
\begin{proposition}
The isometry $\varinjlim C(J_\lambda) \to \varinjlim C(I_\lambda)$
extends to a unique isometry 
$$
L^2(\Jhat,\mu_\Jhat) \to L^2(\Ihat,\mu_\Ihat).
$$ 
\end{proposition}
\begin{proof}
We consider $L^2$-norm for all the spaces. 
The map $\varinjlim C(J_\lambda) \to L^2(\Jhat)$
is isometric, injective, and has a dense image, by considering
the orthogonal basis $E_j$. Since
$L^2(\Jhat)$ is complete, this is a completion of $\varinjlim C(J_\lambda)$.
The map $\varinjlim C(I_\lambda) \to L^2(\Ihat)$
is also isometric, hence injective, and has a dense image
(by Lemma~\ref{lem:density} it is dense with respect to 
supremum norm in $C(\Ihat)$, and since $\Ihat$ is compact, dense
with respect to the $L^2$-norm in $C(\Ihat)$, and thus dense 
in $L^2(\Ihat)$), so $L^2(\Ihat)$ is a completion of 
$\varinjlim C(I_\lambda)$. By the universality of the completion, 
they are isometric. 
\end{proof}

\section{Some examples and relation with Barg-Skriganov theory}
The results of this paper are closely related to the
study by Barg and Skriganov \cite{BARG-SKRIGANOV}. 
They defined the notion of association schemes on a set $X$ with 
a $\sigma$-additive measure. See \cite[Definition~1]{BARG-SKRIGANOV},
where $I$ is denoted by $\Upsilon$. 
Association scheme structure is given by a surjection 
$$
R:X \times X \to I,
$$
with axioms similar to the finite case. The inverse image 
$R^{-1}(i) \subset X \times X$
of $i\in I$ is required to be measurable in $X\times X$,
and intersection numbers $p_{ij}^k$ are defined 
as the measure of the set
$$
\{z\in X \mid R(x,z)=i, R(z,y)=j \}
$$
where $R(x,y)=k$, which is required to be independent of the 
choice of $x,y$. The set $I$ is required to be at most countably 
infinite. This countability assumption seems to be necessary, 
for example to avoid
the case where $p_{ij}^k$ is all zero for any $i,j,k$.

Let $\chi_i(x,y)$ be the indicator function of $R^{-1}(i)$.
They define adjacency algebras (Bose-Mesner algebras in our terminology)
in \cite[Section~3]{BARG-SKRIGANOV} as a set of finite linear
combination of $\chi_i(x,y)$ with complex coefficient, with 
convolution product given by the integration 
$$
(a*b)(x,y):=\int_{X}a(x,z)b(z,y)d\mu(z).
$$
However, it is not clear whether the adjacency algebra is closed
under the convolution product. 
In fact, the authors remarked in the last of \cite[Section~3.1]{BARG-SKRIGANOV}
 ``our arguments in this part are of somewhat heuristic nature. 
We make them fully rigorous for the case of 
association schemes on zero-dimensional 
groups; see Section~8.''
Later in Section~8, they define the adjacency algebra with an 
aid of locally compact zero-dimensional abelian groups and
their duality. Some of the treated objects (compact cases) are examples
of profinite association schemes in this paper.

\subsection{Profinite groups}
Let $G$ be a profinite group, that is, a projective limit of 
finite groups $G_\lambda$ with projective limit topology. 
Absolute Galois groups, arithmetic fundamental groups and $p$-adic
integers are some
examples appear in arithmetic geometry \cite{RIBES-ZALESSKII}.
There are significant amounts of studies on analysis on
abelian profinite groups, see for example \cite{LANG-DYADIC}.

\subsubsection{Schurian profinite association schemes}
We describe Schurian association schemes, 
and then profinite analogues.
When we say an action of a group $G$, it means the left action,
unless otherwise stated.
\begin{lemma}\label{lem:quotient}
Let $X$ be a topological space and $G$ a group
acting on $X$ such that every element of $G$ induces a homeomorphism on $X$.
Then the space of continuous functions $C(X)$ is 
a left $G$-module by $f(x)\mapsto f(g^{-1}x)$. 
Let $G\backslash X$ denote the quotient space
with quotient topology. 
Let 
$$C(X)^G \subset C(X)$$
denote the subspace of functions invariant by the action of $G$.
There is a canonical injection
$$
C(G\backslash X) \hookrightarrow C(X)
$$
associated with the continuous surjection $X \to G\backslash X$.
Then, there is a canonical isomorphism of $\C$-algebras
$$
C(X)^G \to C(G\backslash X), 
$$
and hence we identify them.
\end{lemma}
\begin{proof}
An element $f$ of $C(X)^G$ is a continuous
function which takes one same value
for each $G$-orbit of $X$. Thus, it gives
a mapping $G\backslash X \to \C$. By the universality
of the quotient topology, this is continuous, 
hence lies in $C(G\backslash X)$. Conversely,
a function in $C(G\backslash X)$ gives a 
continuous function in $C(X)$, which is invariant 
by $G$.
\end{proof}

\begin{proposition} (Schurian association scheme)
\label{prop:schurian}

Let $G$ be a finite group, $H$ a subgroup, 
and let $X:=G/H$ be the quotient.
Then, we have a natural bijection
\begin{equation}\label{eq:hecke}
G\backslash (X \times X) \to H\backslash G/H, 
\end{equation}
where $G$ acts diagonally. We define 
$$I:=G\backslash (X \times X).$$
The quotient mapping
$$
R:X \times X  \to I
$$
is an association scheme, which is called 
a Schurian scheme.
\end{proposition}
We give a proof, since descriptions used in 
the proof are necessary for the
profinite case. The equality
(\ref{eq:hecke}) is given by a surjection
$$
(g_1H, g_2H) \mapsto H g_1^{-1}g_2H,
$$
which certainly factors through
$$
G\backslash(g_1H, g_2H) \mapsto Hg_1^{-1}g_2H,
$$
and the converse map is given by 
$$
HgH \mapsto (H, gH),
$$
where the check of the well-definedness is left to the reader.
\begin{lemma}\label{lem:ident}
Let $C(X)^\vee:=\Hom_\C(C(X),\C)$ be the dual space,
which is regarded as a left $G$-module, by letting $g\in G$
act on $\xi\in C(X)^\vee$ by $g(\xi)f=\xi(g^{-1}(f))$.
Then we have an isomorphism of $G$-modules
$$
C(X) \stackrel{\vee}\to C(X)^\vee, \quad \delta_x \mapsto \ev_x
$$
for every $x\in X$, where $\delta_x$ is the indicator function of $x$
and $\ev_x$ is the evaluation at $x$. 
\end{lemma}
\begin{proof}
It is well-known that $\{\delta_x \mid x\in X\}$ 
is a linear basis of $C(X)$, and $\{\ev_x \mid x \in X\}$
is a linear basis of $C(X)^\vee$. 
We compute
$$
g(\delta_x)(y)=\delta_x(g^{-1}y)=\delta_{gx}(y).
$$
On the other hand,
$$
g(\ev_x)(f)=\ev_x(g^{-1}(f))=(g^{-1}(f))(x)=f(gx)=\ev_{gx}(f).
$$
\end{proof}
\begin{proposition}
The identification of a matrix and a linear map is
obtained by an isomorphism of $G$-modules
\begin{equation}\label{eq:matrix}
C(X\times X)=C(X)\otimes C(X)\stackrel{\id \otimes \vee}\to
C(X)\otimes C(X)^\vee \cong \Hom_\C(C(X),C(X)).  
\end{equation}
\end{proposition}
\begin{proof}
An element $A \in C(X\times X)$ is mapped as
\begin{eqnarray*}
A &\mapsto& \sum_{x,y\in X^2}A(x,y)\delta_x \otimes \delta_y \\
  &\mapsto& \sum_{x,y\in X^2}A(x,y)\delta_x \otimes \ev_y.
\end{eqnarray*}
The last element is regarded as an element of $\Hom_\C(C(X),C(X))$
by mapping $f=\sum_{y\in X}f(y)\delta_y \in C(X)$ to 
$$
(\sum_{x,y\in X^2}A(x,y)\delta_x \otimes \ev_y) (f)
=
\sum_{x\in X}\sum_{y\in Y}A(x,y)f(y)\delta_x,
$$
which is the multiplication of a matrix to a vector.
\end{proof}
\begin{proposition}
By taking the $G$-invariant part of (\ref{eq:matrix}),
we have
\begin{equation}\label{eq:hom}
\Hom_G(C(X),C(X))\cong C(X\times X)^G=C(G\backslash (X\times X))
=C(H\backslash G/H),
\end{equation}
where the left $\Hom_G$ denotes the endomorphism
as $G$-modules.
\end{proposition}
\begin{proof}
This follows from (\ref{eq:hecke}), Lemma~\ref{lem:quotient}
and the definition of $G$-homomorphisms. 
\end{proof}
\begin{corollary}
The construction in Proposition~\ref{prop:schurian}
gives an association scheme. 
The endomorphism ring
$\Hom_G(C(X),C(X))$ is isomorphic 
to the Bose-Mesner algebra $A_X$.
\end{corollary}
\begin{proof}
By the construction of $I=G\backslash (X\times X)$, 
the set of the adjacency matrices is closed under the transpose
and contains the identity matrix. 
Their linear span is closed
by the matrix product, since the
adjacency matrices form a linear basis of
(\ref{eq:hom}) and the matrix product corresponds
to the composition of the endomorphism ring $\Hom_G(C(X),C(X))$.
This gives an association scheme with Bose-Mesner
algebra $\Hom_G(C(X),C(X))$.
\end{proof}
\begin{definition}
Let $G$ be a finite group acting on a $\C$-linear space $V$.
By the representation of finite groups, 
$V$ is decomposed into irreducible representations.
For a fixed irreducible representation $W$,
the multiplicity of $W$ in $V$ is the dimension of the vector space
$$
\Hom_G(W,V).
$$
We say $V$ is multiplicity free, if the multiplicity is 
at most one for any irreducible representation $W$.
\end{definition}
\begin{corollary}\label{cor:multiplicity}
The Schurian 
association scheme is commutative, if and only if 
$C(X)$ is multiplicity free. In this case, 
the set of the primitive idempotents is canonically 
identified with the set of the irreducible representations
of $G$ in $C(X)$.
\end{corollary}
\begin{proof}
If the multiplicity is one for each irreducible representation,
Schur's lemma 
implies that $\Hom_G(C(X),C(X))$ is a direct sum of
copies of $\C$ (each copy corresponding to each
irreducible representation in $C(X)$)
as a ring, and hence is commutative. If there is an irreducible
component with multiplicity $m\geq 2$, 
$\Hom_G(C(X),C(X))$ contains a matrix algebra $M_m(\C)$,
and thus is not commutative.
\end{proof}
See \cite{RIBES-ZALESSKII} for the following definition 
and properties (c.f. Section~\ref{sec:def-pro} above).
\begin{definition}
Let $G_\lambda (\lambda\in \Lambda)$ be a projective system 
of surjections of finite groups. A profinite group
is a topological group which is isomorphic to
$$
\Ghat:=\varprojlim G_\lambda,
$$
where each $G_\lambda$ has the discrete topology.
Because the projection $\pr_\lambda:\Ghat \to G_\lambda$ is
surjective and continuous, we identify $G_\lambda$ 
with $\Ghat/N_\lambda$, where $N_\lambda$ is
the kernel of the $\pr_\lambda$, which is a clopen
normal subgroup of $\Ghat$, since $G_\lambda$ is finite and 
discrete. By Tychonoff's theorem, $\Ghat$ is a compact Hausdorff group, with a 
basis of open neighborhood of the unit element $e$ consisting of $N_\lambda$.
Being Hausdorff, the intersection of $N_\lambda$ is 
$\{e\}$.
\end{definition}

We prepare a well-known lemma for describing profinite
Schurian association schemes.
\begin{lemma}\label{lem:prep}
Let $\Ghat= \varprojlim G_\lambda = \varprojlim \Ghat/N_\lambda$ 
be a profinite group. 
Let $H$ and $K$ be 
closed subgroups of $\Ghat$.
Give a quotient topology on $H\backslash \Ghat/K$
by 
$$
q: \Ghat \to H\backslash \Ghat/K.
$$
\begin{enumerate}
 \item \label{enum:image}
The image of $K$ in $G_\lambda$ is $KN_\lambda$
(which is a union of cosets of $N_\lambda$, and hence
a subset of $G_\lambda$).
The same statement holds for $H$.
 \item \label{enum:bij}
We have the following commutative diagram 
of continuous surjections, where $\alpha$ is a bijection.
\begin{equation}\label{eq:triangle}
\begin{array}{ccc}
 H\backslash \Ghat /K & & \\
  \downarrow & \searrow & \\
 HN_\lambda\backslash \Ghat/KN_\lambda &\stackrel{\alpha}{\to}&
 HN_\lambda\backslash G_\lambda/KN_\lambda.
\end{array} 
\end{equation}
\item \label{enum:open}
The $q:\Ghat \to H\backslash \Ghat/K$ is an open morphism,
and $H\backslash \Ghat/K$ is compact Hausdorff.
\item \label{enum:proj}
The finite sets $HN_\lambda\backslash G_\lambda/KN_\lambda$
constitute a projective system, with 
a canonical homeomorphism
$$
H\backslash \Ghat/K\stackrel{\sim}{\to}
\varprojlim HN_\lambda\backslash G_\lambda/KN_\lambda.
$$
\end{enumerate}
\end{lemma}
\begin{proof}

(\ref{enum:image}) When we consider $G_\lambda$ 
as a quotient of $\Ghat$, $G_\lambda$ is a set of 
cosets $gN_\lambda$.
The image of $K$ is the union of $kN_\lambda$
for $k\in K$, which is nothing but $KN_\lambda$.

(\ref{enum:bij}) Every set in the diagram is a quotient
of $\Ghat$ with various equivalence relation.
The vertical arrow is obviously well-defined.
The mapping $\alpha$ is bijective, since as a quotient of 
$\Ghat$, they are the same. 
The slanting arrow is merely the composition of the two.
We recall that $\Ghat \to H\backslash \Ghat /K$
is a quotient map. Thus, every map in the diagram is continuous,
since the maps from $\Ghat$ are continuous
(where the mappings to the two sets at the bottom of the diagram factors 
through $\Ghat/N_\lambda$, which is finite and discrete).

(\ref{enum:open})
The open sets $sN_\lambda$ for $s\in \Ghat$ and various 
$\lambda$ constitute an open basis of $\Ghat$.
The image of $sN_\lambda$ in $H\backslash \Ghat/K$
is $HsKN_\lambda$ (since $N_\lambda$ is normal),
whose inverse image in $\Ghat$ is again $HsKN_\lambda$, 
which is open. Thus, by the definition of quotient 
topology, $HsKN_\lambda$ is open in $H\backslash \Ghat/K$,
and hence $q$ is open.

Take two distinct elements $HgK$, $Hg'K$ in $H\backslash \Ghat/K$.
These are compact disjoint subsets of the Hausdorff space
$\Ghat$. 
It is well known that 
there are an open set $O_a\subset \Ghat$ containing $HgK$ and 
an open set $O_b \supset Hg'K$ that separate 
$HgK$ and $Hg'K$. By (\ref{enum:separate}) of Lemma~\ref{lem:density}
for the union of $HgK$ and $Hg'K$,
we may take $\mu$ such that
open sets of the form $sN_\mu\subset O_a$ with $sN_\mu \cap HgK \neq \emptyset$ cover $HgK$
and open sets $sN_\mu \subset O_b$ with $sN_\mu \cap Hg'K \neq \emptyset$ cover $Hg'K$, that is, 
$HgKN_\mu\subset O_a$ and $Hg'KN_\mu \subset O_b$ are disjoint. 
Since $q$ is open, $HgKN_\mu$ and $Hg'KN_\mu$ 
are open sets in $H\backslash \Ghat/K$ which have 
no intersection and are neighborhoods of $HgK$, $Hg'K$
(which are considered as elements of $H\backslash \Ghat/K$),
respectively. Thus $H\backslash \Ghat/K$ is Hausdorff.

(\ref{enum:proj}) If we see these sets
as quotients of $\Ghat$ as in (\ref{eq:triangle}), it is clear that these
form a projective system. By the universality of 
projective limit, we have a continuous map
$$
f:H\backslash \Ghat/K \to \varprojlim HN_\lambda\backslash G_\lambda/KN_\lambda
$$
by (\ref{enum:bij}). The image of $f$ is dense, 
since for any element $\xi$ in the projective limit
and its open neighborhood, there is a smaller
open neighborhood which has the form
$\pr_\lambda^{-1}\pr_\lambda(\xi)$ for some $\lambda$,
since these form a basis of open neighborhoods of $\xi$.
Because
$$
H\backslash \Ghat/K \to HN_\lambda\backslash G_\lambda/KN_\lambda
$$
is surjective, the density of the image follows.
Since $H\backslash \Ghat/K$ is compact (being a quotient of
a compact set $\Ghat$), the image is compact and closed,
hence $f$ is surjective. Since $H\backslash \Ghat/K$
is Hausdorff, any two distinct points $HgK$, $Hg'K$
have disjoint open neighborhoods
$HgKN_\lambda$, $Hg'KN_\lambda$, respectively, 
as proved above, hence their images
in $HN_\lambda\backslash G_\lambda/KN_\lambda$ are distinct,
which implies the injectivity of $f$. 
By the compactness of $H\backslash \Ghat/K$, 
$f$ is closed, hence $f^{-1}$ is continuous,
thus a homeomorphism.
\end{proof}

\begin{proposition} (profinite Schurian association scheme) 
\label{prop:profin-Schurian}
Let $\Ghat= \varprojlim G_\lambda = \varprojlim \Ghat/N_\lambda$ 
be a profinite group.
We denote by $G_\lambda$ the quotient $\Ghat/N_\lambda$.
For a closed subgroup $H<G$, its image in $G_\lambda$
is $H_\lambda:=HN_\lambda$.
Then we have
a projective system of Schurian association schemes
$$
X_\lambda:=G_\lambda/H_\lambda
$$
and
$$
I_\lambda:=H_\lambda \backslash G_\lambda/H_\lambda.
$$
These systems give a profinite association scheme, 
which may be non-commutative. We call this type of
profinite association scheme Schurian.
We have canonical homeomorphisms
$$
\Ghat/H \stackrel{\sim}{\to} \Xhat = \varprojlim X_\lambda
=\varprojlim \Ghat/HN_\lambda
$$
and
$$
H\backslash \Ghat/H
\stackrel{\sim}{\to} \Ihat = \varprojlim I_\lambda=
\varprojlim H_\lambda \backslash \Ghat/H_\lambda.
$$
\end{proposition}
\begin{proof}
Everything is proved in Proposition~\ref{prop:schurian} 
and Lemma~\ref{lem:prep},
except that these mappings give
surjective morphisms of association schemes.
Take $\lambda\geq \mu$. 
Since the surjectivity is proved 
in (\ref{enum:bij}), it suffices to show the commutativity of
$$
\begin{matrix}
X_\lambda \times X_\lambda &\to& I_\lambda \\
\downarrow & & \downarrow \\
X_\mu \times X_\mu &\to& I_\mu,
\end{matrix}
$$
which is easy to check: an element $(gHN_\lambda, g'HN_\lambda)$
at the left top 
is mapped to the right $Hg^{-1}g'HN_\lambda$, then to the
bottom $Hg^{-1}g'HN_\mu$. This is the same via the left bottom corner.
\end{proof}

\begin{proposition}\label{prop:commute-hat}
Let $\Ghat/H$ be a Schurian profinite association scheme as
in \ref{prop:profin-Schurian}. Suppose that
each finite association scheme $X_\lambda$ is commutative.
Then, $\Jhat$ is the set of irreducible representations of $\Ghat$
appearing in $C(X_\lambda)$ for some $\lambda$, where
the multiplicity is at most one for any $C(X_\lambda)$.
\end{proposition}
\begin{proof}
We have an inductive system of injections
$C(X_\lambda \times X_\lambda)$,
which gives via (\ref{eq:matrix}) an inductive system 
of injections
$$
\Hom_\C(C(X_\mu),C(X_\mu)) \to \Hom_\C(C(X_\lambda),C(X_\lambda))
$$
for $\lambda \geq \mu$. To obtain 
$A_{X_\mu}$ and $A_{X_\lambda}$, we take
$\Ghat$-invariant part (\ref{eq:hom})
$$
\Hom_{G_\mu}(C(X_\mu),C(X_\mu)) \to 
\Hom_{G_\lambda}(C(X_\lambda),C(X_\lambda))
$$
(note that there is a canonical surjection $\Ghat \to G_\lambda$, hence
a $G_\lambda$-module is a $\Ghat$-module.)
By the surjectivity of $\Ghat \to G_\lambda$, an 
irreducible representation in $C(X_\mu)$ of $G_\mu$
lifts to an irreducible representation in $C(X_\lambda)$
of $G_\lambda$. 

We assume that every finite association 
scheme appeared is commutative. Then, the multiplicity of 
each irreducible representation $V\subset C(X_\lambda)$ is one
by Corollary~\ref{cor:multiplicity} for any $\lambda$. 
Thus, the partial surjection 
given in Corollary~\ref{cor:J} is in this case the inverse to the 
injection $J_\mu \to J_\lambda$, induced by $G_\lambda \to G_\mu$, 
and 
$$
\Jhat = \varinjlim J_\lambda
$$
is identified with the set of irreducible representations of
$\Ghat$ appeared in some $C(X_\lambda)$.
\end{proof}

Indeed, the above irreducible representations
are exactly the irreducible representations
of $\Ghat$ appearing in $C(\Ghat/H)$.
To show this, we recall a well known 
``no small subgroups'' lemma.
\begin{lemma}\label{lem:no-small}
Let $G$ be a Lie group over $\R$ or $\C$. 
Then, there is an open neighborhood of the unit $e$,
which contains no subgroup except $\{e\}$.
\end{lemma}
\begin{proof}
We consider the Lie algebra $\mathfrak{g}$ of $G$. Give an 
arbitrary (Hermitian or positive definite)
metric to $\mathfrak{g}$. It is known that 
$$\exp:\mathfrak{g}\to G$$
is homeomorphism when restricted to a small enough neighborhood
of $0$. We take the homeomorphic neighborhoods $U$ of $e\in G$
and $V$ of $0\in \mathfrak{g}$. We may assume that $V$ is bounded
with respect to the metric.
Assume that $U$ contains non-trivial subgroup of $G$.
Take an $e\neq g\in G$ in the subgroup. Consequently, for any integer $n$,
$g^n\in U$.
If $h=\log(g)\in V \subset \mathfrak{g}$, $nh=\log(g^n)$ corresponds to $g^n$
and hence $nh \in V$ for any $n$. However, $V$ is bounded, 
which contradicts the assumption.
\end{proof}
\begin{corollary}\label{cor:finite-image}
Let $\Ghat$ be a profinite group, and $f:\Ghat \to G$
a group morphism to a real or complex Lie group $G$.
Then, $\Ker(f)$ is an open normal subgroup of $\Ghat$,
and the image $f(\Ghat)$ in $G$ is a finite subgroup of $G$.
\end{corollary}
\begin{proof}
Take a small enough neighborhood of $U$ in Lemma~\ref{lem:no-small}.
Consider $H=f^{-1}(U)\subset \Ghat$, which is open in $\Ghat$.
Thus, 
there exists $\lambda$ such that the open normal subgroup $N = N_\lambda$ is contained in $H$.
The image $f(N)<G$ is a subgroup in $U$, which must be $\{e\}\subset G$ by definition of $U$.
Thus, $\Ker(f)$ contains $N$, and is a union of cosets
of $N$, hence an open normal subgroup of $\Ghat$.
Since $\Ghat$ is compact and $\Ker f$ is open, $f(\Ghat)\cong\Ghat/\Ker f$ is a finite group.
\end{proof}

\begin{proposition}\label{prop:compact-repr}
Let $G$ be a compact Hausdorff group, and $H$ a closed subgroup.
Consider the representation of $G$ on $C(G/H)$.
Then, each irreducible component of $C(G/H)$
is finite dimensional. 
\end{proposition}
\begin{proof}
This follows from a variant of the Peter-Weyl theorem,
see Takeuchi\cite[Section~1]{TAKEUCHI}.
\end{proof}

\begin{proposition}\label{prop:comp-kurihara}
Let $\Ghat$ be a profinite group and $H$ a closed subgroup.
Consider the profinite association scheme
arising from $\Ghat/H$. 
Let $J(\Ghat/H)$ be the set of the irreducible representations 
of $\Ghat$ on $C(\Ghat/H)$.
Then, all finite 
association schemes $X_\lambda$ in Proposition~\ref{prop:profin-Schurian}
are commutative, if and only if the representation $C(\Ghat/H)$
is multiplicity free. 
In this case, there is a canonical bijection
\begin{equation}\label{eq:Jhat-schur}
f:J(\Ghat/H) \stackrel{\sim}{\to} \Jhat=\varinjlim J_\lambda. 
\end{equation}

\end{proposition}
\begin{proof}
Put $\Xhat:=\Ghat/H$.
Let $V\subset C(\Xhat)$ be an irreducible representation 
of $\Ghat$.
By Proposition~\ref{prop:compact-repr}, $V$ is finite dimensional.
By Lemma~\ref{lem:no-small}, the representation factors through
$$
G \to G/N_\lambda \to \GL(V)
$$
for some $\lambda$. By Lemma~\ref{lem:quotient}, this means that $V\subset C(\Xhat)$
lies in 
$$
V \subset C(N_\lambda\backslash \Xhat)=C(X_\lambda) \subset C(\Xhat).
$$
Thus, $C(\Xhat)$ is multiplicity free, if and only if
$C(X_\lambda)$ is multiplicity free for every $\lambda\in \Lambda$, 
in other words, 
if and only if every $X_\lambda$ is a commutative association scheme,
by Proposition~\ref{prop:commute-hat}. In the commutative case,
the following observation shows that every irreducible
component of $C(\Xhat)$ appears in $C(X_\lambda)$ for some $\lambda$.
Because the multiplicity is one,
for any $V$ above, there is a unique
representation isomorphic to $V$ in $\varinjlim C(X_\lambda)$. 
This gives a map $f$ in 
(\ref{eq:Jhat-schur}). 
The injectivity follows from that $C(\Xhat)$ is multiplicity-free. 
The surjectivity follows from 
$C(X_\lambda)\subset C(\Xhat)$
is a $\Ghat$-submodule, and any $G_\lambda$-irreducible component
of $C(X_\lambda)$ is a $\Ghat$-irreducible component
since $\Ghat \to G_\lambda$ is surjective. 
\end{proof}
\begin{remark}
Kurihara-Okuda\cite{KURIHARA-OKUDA} gives the notion of
Bose-Mesner algebra for general homogeneous space $G/H$,
where $G$ is compact Hausdorff and $H$ is a closed subgroup.
In the case where $G$ is a profinite group and $C(G/H)$
is multiplicity free, their construction coincides with
our definition, by Proposition~\ref{prop:comp-kurihara}.
\end{remark}
\subsubsection{Profinite abelian groups}

For any finite abelian group $X$, we consider its thin scheme.
Namely, $R:X\times X \to I=X$ with $(x,y) \mapsto y-x$.
Then, the $i$-th adjacency matrix $R^{-1}(i)$ is the 
representation matrix of addition of $i$ on $C(X)$,
and the Bose-Mesner algebra is isomorphic to the 
group ring $C(X)$, with usual Hadamard product and
the convolution product 
$$
(f * g)(x)=\frac{1}{\#X}\sum_{y}f(x-y)g(y)
$$ 
(the factor $\frac{1}{\#X}$ is not the standard, incorporated
to be compatible with Theorem~\ref{th:convolution}).
Primitive idempotents are exactly the characters
$\xi \in \Xcheck$, where $\Xcheck \subset C(X)$ is the group of 
characters of $X$. 

For a profinite abelian group $\Xhat$, it is a projective
limit of finite abelian groups $X_\lambda$. We have an 
inductive system of the character groups $J_\lambda:=\check{X_\lambda}$.
For $\lambda \geq \mu$, 
the partial map $p_{\lambda,\mu}:J_\lambda \parto J_\mu$ is given by
the injection $i_{\mu,\lambda}:J_\mu \to J_\lambda$, where the 
domain of $p_{\lambda,\mu}$ is the image of $i_{\mu,\lambda}$,
and in the domain $p_{\lambda,\mu}$ is the inverse of $i_{\mu,\lambda}$.
Thus, $\Ihat$ is isomorphic to $\Xhat$, $\Jhat$ is the
inductive limit of the character groups $\check{X_\lambda}$.
The Bose-Mesner algebra is isomorphic to $C_c(\Jhat)$,
the space of finite linear combinations of characters of $\Xhat$,
which is isomorphic to $C_{lc}(\Ihat)$, the space of locally constant
functions on $\Ihat \simeq \Xhat$. This example is outside of the scope of Barg-Skriganov
theory, if $\Ihat \simeq \Xhat$ is uncountable.
\begin{remark}
Since $\Jhat$ is discrete, $C_c(\Jhat)$ has a natural (and orthogonal)
basis consisting of $E_j$ ($j\in \Jhat$), 
but for the canonically 
isomorphic vector space $C_{lc(\Ihat)}$,
we do not have a natural basis, since $\Ihat$
is compact with (possibly)
uncountable cardinality. This makes a definition 
of eigenmatrices difficult.
\end{remark}

\subsection{Kernel schemes}
Another example of profinite association schemes is
a projective system of the kernel schemes. 
The kernel schemes are finite association schemes 
defined in Martin-Stinson\cite{MARTIN-STINSON} to study $(t,m,s)$-nets.
The $(t,m,s)$-nets are introduced by Niederreiter for quasi-Monte Carlo
integrations, see \cite{niederreiter:book}.

We recall the kernel scheme, with a slight modification
on the notation for $I$, to make a description as a projective
system easier. 
\begin{definition}\label{def:kernel}
Let $n$ be a positive integer, and $V$ be a finite set of 
alphabet with cardinality $v\geq 2$. Let $X_n$ be $V^n$,
and $I_n:=\{1,2,\ldots,n\}\cup\{\infty\}$.
Define $R_n:X_n\times X_n \to I_n$ as follows.
Let $x=(x_1,x_2,\ldots,x_n)$ and $y=(y_1,y_2,\ldots,y_n)$ be
elements of $X_n$. Let $R(x,y)$ be the smallest index $i$
for which $x_i\neq y_i$. If $x=y$, then $R(x,y)=\infty$.
This is a symmetric (and hence commutative) 
association scheme, with $R^{-1}(\infty)$ being the identity
relation. This is called a kernel scheme, and denoted by
$\overrightarrow{k(n,v)}$.
\end{definition}
A direct computation of the intersection numbers
shows that this is a commutative association scheme,
as shown in \cite{MARTIN-STINSON}.
We shall give an indirect proof
in Corollary~\ref{cor:ker-as}. (This is only for the self-containedness.)
\begin{proposition}\label{prop:pro-k}
The kernel schemes $\overrightarrow{k(n,v)}$ 
forms a projective system, where $X_{n+1} \to X_n$ is given
by deleting the right most component, 
and $I_{n+1} \to I_n$ is given by $m\mapsto m$ for $m\leq n$,
and the both $n+1,\infty$ are mapped to $\infty$.
Then $\Xhat=V^{\N_{>0}}$ holds, and $\pr_n:\Xhat \to X_n$ is obtained by
taking the left (first) $n$ components, while $\Ihat$ is 
$\N_{>0} \cup \{\infty\}$, with topology obtained by: each $n\in \N_{>0}$
is clopen, and an open set containing $\infty$ is a complement 
of a finite subset of $\N_{>0}$, that is, $\Ihat$ is the one-point 
compactification of $\N_{>0}$.
We call this a pro-kernel scheme and denote by $\overrightarrow{k(\infty, v)}$.
\end{proposition}
\begin{proof} It is easy to check that these make a 
projective system of association schemes. The topology
of $\Ihat$ comes from the definition of projective limit topology.
\end{proof}

\begin{remark} In this case $\Ihat$ is countable, and this is 
a special case of infinite association schemes examined in detail,
named metric schemes, by Barg-Skriganov \cite[Section~8]{BARG-SKRIGANOV}, 
for which the adjacency algebra is defined.
In particular, in (8.33) they showed that it is an inductive
limit of finite dimensional adjacency algebra, which indeed
coincides with our definition.
\end{remark}
\begin{remark} This is merely an observation. For the 
set of $p$-adic integers $\Z_p$, we have a valuation 
$v:\Z_p \to \N\cup \{\infty\}$ giving an ultrametric.
The above pro-kernel scheme is isomorphic to this metric,
up to the difference by one on the metric,
i.e., 
$$
R:\Z_p \times \Z_p \to \N\cup \{\infty\}, (x,y)\mapsto v(x-y).
$$
The same holds for a ring of formal power series 
$\F_q[[t]]$, where $\F_q$ denotes the finite field of $q$ elements.
The both yield the isomorphic pro-kernel schemes if $p=q$.
\end{remark}
To see $\Jhat$ of the pro-kernel scheme, we need to work 
with primitive idempotents. All necessary results are in
Martin-Stinson\cite{MARTIN-STINSON}, but because of the choice of 
the index $I$ and normalization of the products necessary
for making the projective system, the obtained constants 
slightly differ. To help the readers' understandings, we 
recall the methods of Martin-Stinson.
For the chosen integer $v\geq 2$,
let $V:=Z_v$ be the additive group $\Z/v$. Then 
$X=Z_v^n$ is an additive group, and the kernel scheme
is a translation scheme \cite[\S2.10]{BROUWER}.
\begin{definition}
Let $X$ be a finite abelian group. A translation scheme is
an association scheme $R:X\times X \to I$, which
factors through $X\times X \to X, (x,y)\mapsto y-x$,
namely, there is an $S:X \to I$ such that $R(x,y)=S(y-x)$. 
\end{definition}

\begin{definition}\label{def:bot} 
Return to the kernel scheme.
We define
$\bbot:X_n \to I_n$ by 
$$
\bbot(x_1,\ldots,x_n):=\min\{i \mid x_i\neq 0\},
$$
and $\bbot(0,\cdots,0)=\infty$. Then we have
$$
R_n:X_n\times X_n \to I_n, \quad (x,y) \mapsto \bbot(x-y).
$$
By $S:=\bbot$, the kernel scheme
is a translation scheme.
\end{definition}

Let $(X,R,I)$ be a translation scheme. 
For $i\in I$, let $S_i\subset X$ be $S^{-1}(i)$ ($S$ for the sphere).
Let $\Xcheck$ be the dual group of $X$. Then, 
$\xi \in \Xcheck$ is an orthonomal basis of
$C(X)$ under the normalized inner product
in Definition~\ref{def:inner-product}.
Let $A_i$ denote the $i$-th adjacency matrix, and 
$\chi_i:X \to \{0,1\}$ be the indicator function for $S_i$.
Our action in Theorem~\ref{th:convolution} shows that for $x\in X$
\begin{eqnarray*}
(A_i\bullet \xi)(x) &=&
\frac{1}{\#X}\sum_{y\in X}A_i(x,y)\xi(y) \\
&=&
\frac{1}{\#X}\sum_{y\in X}\chi_i(y-x)\xi(y) \\
&=&
\frac{1}{\#X}\sum_{y\in x+S_i}\xi(y) \\
&=&
\frac{1}{\#X}\sum_{a\in S_i}\xi(x+a) \\
&=&
\frac{1}{\#X}(\sum_{a\in S_i}\xi(a))\xi(x).
\end{eqnarray*}
Thus $\xi$ is an eigenvector with eigenvalue 
\begin{equation}\label{eq:eigen}
\frac{1}{\#X}(\sum_{a\in S_i}\xi(a)). 
\end{equation}
The primitive idempotents are projectors to the 
common eigen spaces of $A_i$ for all $i\in I$.
Let us assume $X=Z_v^n$, and then
$\Xcheck=\Zvcheck^n$, where the product of $\Zvcheck$ is
written multiplicatively with unit $1$. Thus, 
\begin{equation}\label{eq:eval}
(\xi_1,\ldots,\xi_n)(x_1,\ldots,x_n)=\prod_{i=1}^n \xi_i(x_i). 
\end{equation}
\begin{definition}\label{def:ttop}
For $\xi=(\xi_1,\ldots,\xi_n)\in \Xcheck$, 
define $\ttop(\xi)$ as the max $i$ such that $\xi_i\neq 1$.
If $\xi=1$, then $\ttop(\xi)=0$. Thus $\ttop: \Xcheck \to \{0,1,\ldots,n\}$. 
We shall denote $J_n:=\{0,1,\ldots,n\}$, because we shall soon show
a natural correspondence between the set of primitive idempotents
and $J_n$.
\end{definition}
\begin{lemma}\label{lem:eigen}
For $\xi$ with $j=\ttop(\xi)$, 
we denote by $p_i(j)$ the eigenvalue of $A_i$ for $\xi$.
Then
$$
p_i(j)=
\begin{cases}
v^{-n} & \mbox{ if } i=\infty, \\
(v-1)v^{-i} & \mbox{ if } j<i<\infty, \\
-v^{-i} & \mbox{ if } j=i, \\
0 & \mbox{ if } j>i. \\
\end{cases}
$$
Thus, 
for any $A_i$, the eigenvalue of $\xi\in \Xcheck$ depends only on $\ttop(\xi)$,
and if $\ttop(\xi)\neq \ttop(\xi')$, then there is an $i$
such that $A_i$ has different eigenvalues for $\xi$ and $\xi'$.
\end{lemma}
\begin{proof}
Recall that $I=\{1,2,\ldots,n\}\cup\{\infty\}$.
Put
$$
B_i:=\bigcup_{i \leq \ell \leq \infty} S_\ell.
$$
($B_i$ is the ball consisting of all the elements
$(x_1,\ldots,x_n)$ with $x_\ell=0$ for all $\ell<i$,
with the cardinality $v^{n-i+1}$ for $i<\infty$ and $1$ for
$i=\infty$.)
Fix $i$ and $\xi\in \Xcheck$ with $\ttop(\xi)=j$, then
by (\ref{eq:eigen}) we have
\begin{eqnarray*}
\tau(i,j)\xi &:=&
\sum_{i \leq \ell \leq \infty} A_i \bullet \xi \\
&=& \frac{1}{v^n}(\sum_{a\in B_i}\xi(a)) \xi
\\
&=& 
\begin{cases}
v^{-n}\xi  & \mbox{ if } i=\infty, \\
v^{-i+1}\xi & \mbox{ if } j<i<\infty, \\
0 & \mbox{ if } j\geq i. \\
\end{cases}
\end{eqnarray*}
The last equality follows from that for $i=\infty$, $B_\infty=\{0\}$
and hence the summation is $\sum_{a\in B_i}\xi(a)=1$, and 
if $j<i<\infty$, then (\ref{eq:eval}) is always $1$ and 
the cardinality of $B_i$ is $v^{n-i+1}$, while if $j\geq i$,
$\xi(a)$ depends on $a$ and the summation over $a \in B_i$
is zero, since $\xi$ is a non-trivial character 
on the abelian group $B_i$.
Thus, the eigenvalue of $A_i$ for $\xi$ is the summation 
over $S_i=B_i\setminus B_{i+1}$, that is, 
$$
\begin{cases}
v^{-n} & \mbox{ if } i=\infty, \\
(\tau(i,j)-\tau(i+1,j))=(v-1)v^{-i} & \mbox{ if } j<i<\infty,\\
(\tau(i,j)-\tau(i+1,j))=-v^{-i} & \mbox{ if } i=j, \\
0 & \mbox{ if } j>i,
\end{cases}
$$
as desired.
\end{proof}
\begin{corollary}\label{cor:ker-as}
The kernel scheme $\overrightarrow{k(n,v)}$
is a commutative association scheme.
A primitive idempotent is the projection 
to the subspace $C(X)_j$ of $C(X)$ spanned by $\{\xi \mid \ttop(\xi)=j\}$
for each $j\in J_n$, which identifies the set of primitive
idempotents with $J_n$.
\end{corollary}
\begin{proof}
Recall that we did not use the property of
association schemes for the kernel schemes so far.
Every $A_i$ has $C(X)_j$ as an eigenspace with
an eigenvalue. Since $C(X)$ is a direct
sum of $C(X)_j$ for $j\in J_n$, this implies that
the matrices $A_i$ mutually commute. We consider
the algebra generated by $A_i$ by the matrix product.
Each element of this algebra is determined by 
the $n+1$ eigenvalues corresponding to $J_n$.
Thus, the dimension of the algebra
as a vector space does not exceed
$n+1$, but $A_i$'s are $n+1$ linearly 
independent matrices, hence their linear
combination is closed under the matrix
multiplication. The rest of the axioms
of association schemes are easy to verify.
The description of primitive idempotents
immediately follows.
\end{proof}

\begin{corollary}
The multiplicity $m_j$ is $1$ for $j=0$ and $m_j=(v-1)v^{j-1}$ for $j>0$.
The valency $k_i$ is $1$ for $i=\infty$ and $k_i=(v-1)v^{n-i}$
for $i<\infty$.
\end{corollary}
\begin{proof}
The multiplicity is the number of $\xi$ with $\ttop(\xi)=j$,
which implies the first statement. The valency is the cardinality 
of the sphere $S_i$, which implies the second.
\end{proof}
\begin{corollary}\label{cor:by-p}
For $j\in J_n$, 
the projector $E_j$ is given by
\begin{eqnarray*}
E_j(x,y)&=&
\sum_{\xi\in \Xcheck, \ttop(\xi)=j} \xi(x)\overline{\xi(y)} \\
&=&
\sum_{\xi\in \Xcheck, \ttop(\xi)=j} \xi(x-y).
\end{eqnarray*}
We have 
$$
A_i=\sum_{j\in J_n} p_i(j)E_j
$$
where $p_i(j)$ is defined in Lemma~\ref{lem:eigen}.
\end{corollary}
\begin{proof}
The first half describes the projection to the eigenspace
corresponding to $j$ by an orthonomal basis. The second
is because $A_i$ has the eigenvalue $p_i(j)$ for that eigenspace.
\end{proof}
For $x-y \in S_i$, we define
\begin{eqnarray*}
q_j(i)&:=& E_j(x,y) = 
\sum_{\xi\in \Xcheck, \ttop(\xi)=j} \xi(x-y) \\
&=&
\begin{cases}
1 & \mbox{ if } 0=j \\
(v-1)v^{j-1} & \mbox{ if } 0<j<i \leq \infty \\
-v^{j-1} & \mbox{ if } j=i \\
0 & \mbox{ if } j>i.
\end{cases}
\end{eqnarray*}
This computation is dual to that for $p_i(j)$, so omitted.
\begin{corollary}\label{cor:by-q}
$$
E_j=\sum_{i \in I} q_j(i)A_i.
$$
\end{corollary}
\begin{proof}
We evaluate the both sides at $(x,y)$.
The left hand side is $q_j(i)$ for $x-y\in S_i$,
which is equal to the right hand side, since
$A_i(x,y)=\chi_i(x-y)$.
\end{proof}

\begin{corollary}\label{cor:Jhat-kernel}
In $\overrightarrow{k(\infty,v)}$, the set $\Jhat$
is the inductive limit of $J_n \to J_{n+1}$, 
and hence naturally identified with $\N$. 
\end{corollary}

Now we pass to the pro-kernel scheme, i.e., take
the inductive limit of $A_{\overrightarrow{k(n,v)}}$
to obtain 
$A_{\Xhat}=A_{\overrightarrow{k(\infty,v)}}$. 
\begin{proposition}\label{prop:pro-k-AE}
Let $X_n$ denote the kernel scheme
$\overrightarrow{k(n,v)}$ with $I_n$ and
primitive idempotents $J_n$, and $\Xhat$ the pro-kernel scheme
$\overrightarrow{k(\infty,v)}$ with $\Ihat$
and the primitive idempotents $\Jhat$. For $i\in \Ihat\setminus\{\infty\}$,
we denote by $A_i$ the image of $\delta_i \in C_{lc}(\Ihat)$
by the isomorphism $C_{lc}(\Ihat) \to A_\Xhat$. Note that
$\delta_\infty \notin C_{lc(\Ihat)}$.
Through $A_{X_n} \to A_{\Xhat}$, $E_j$ is mapped to $E_j$,
and $A_i$ for $i\in \{1,2,\ldots,n\}$ is mapped to $A_i$,
while $A_{\infty}\in A_{X_n}$ is mapped to an element corresponding to the
indicator function
$\chi_{\{n+1,\ldots,\infty\}}\in C_{lc}(\Ihat)$.
It is convenient to denote the corresponding element
in $A_\Xhat$ by the symbol
$$
\sum_{i\in \{n+1,\ldots,\infty\}} A_i,
$$
since
$$
\chi_{\{n+1,\ldots,\infty\}}
=
\sum_{i\in \{n+1,\ldots,\infty\}} \delta_i
$$
holds. (Note that $A_\infty$ does not exist in $A_\Xhat$.)
\end{proposition}
\begin{proof}
According to Proposition~\ref{prop:pro-k} 
we have $\Ihat=\N_{>0}\cup \{\infty\}$ and $\Jhat=\N$ by Corollary~\ref{cor:Jhat-kernel}.
Since $J_n \to \Jhat$ is an injection, $E_j\in A_{X_n}$
is mapped to $E_j \in A_\Xhat$ (see Proposition~\ref{prop:orthogonal}).
The projection $\Ihat \to I_n$ is obtained by mapping $i \mapsto i$
for $1\leq i\leq n$, and $i\mapsto \infty$ for $i>n$.
Since the indicator function 
$\delta_i\in C(I_n)$ is mapped to $\delta_i \in C_{lc}(\Ihat)$ for $1\leq i\leq n$,
$A_i$ is mapped to $A_i$. The description of the image of $A_\infty$
is by the definition of $A_{X_n}\to A_{\Xhat}$, which comes from
$\Ihat \to I_n$.
\end{proof}
Everything is well-defined, except $i_0=\infty \in \Ihat$.
In fact, the operators $A_i$, $E_j$,
the values $p_i(j)$, $q_j(i)$, $\mu_\Jhat(j)=m_j$, 
$\mu_\Ihat(i)=k_i/\#X=(v-1)v^{-i}$ are all well-defined,
except that $A_\infty$ and $p_\infty(j)$
are not well-defined, since there is no
corresponding element to $A_\infty$ in $A_\Xhat$. Indeed, $A_\infty$ would 
correspond to 
the indicator function of $\infty \in \Ihat$,
which is not continuous and hence is outside $A_\Xhat$.
This forces to choose something which replaces $A_\infty$.
One possible way is to replace it with $E_0$ because
$$
E_0=\sum_{i \in \Ihat} A_i
$$
holds as a function (constant value $1$) on $\Ihat$. 
then 
$$
\{A_i \mid i\in \N_{>0}\}\cup \{E_0\}
$$
is a linear basis of $C_{lc}(\Ihat)$, although
this is not an orthogonal basis. This is the method
chosen in \cite[(8.3), (8.9)]{BARG-SKRIGANOV} to replace 
an infinite sum with a finite sum 
($\alpha_0$ in (8.8) in their paper is nothing but $E_0$),
which gives a precise definition of the adjacency algebra.
In our interpretation in Proposition~\ref{prop:pro-k-AE}, 
the expression in Corollary~\ref{cor:by-p} becomes
$$
A_i=\sum_{j\in \Jhat} p_i(j)E_j 
$$
which is a finite sum (see Lemma~\ref{lem:eigen}) and holds for $i\neq \infty$
(since the equality holds in $A_{X_i}$), except
the problem that for $i=\infty$, the value $p_\infty(j)=v^{-n}$
converges to $0$, while the summation is over infinitely many $E_j$.
In Corollary~\ref{cor:by-q} the expression
$$
E_j=\sum_{i \in \Ihat} q_j(i)A_i
$$
seems an infinite sum, but $q_j(i)=(v-1)v^{j-1}$ is constant for $i>j$,
and thus the right hand side lies in $C_{lc}(\Ihat)$,
or more precisely, the seemingly infinite sum is a finite sum
by the usage of the symbol in Proposition~\ref{prop:pro-k-AE},
and the equality does hold (since it holds in $A_{X_j}$.)
This may be translated into a finite summation formula for $j\geq 1$:
\begin{eqnarray*}
E_j&=&\sum_{i \leq j} (q_j(i)-(v-1)v^{j-1})A_i + (v-1)v^{j-1}E_0 \\
&=&(\sum_{i<j} -(v-1)v^{j-1}A_i) -v^{j}A_j + (v-1)v^{j-1}E_0.
\end{eqnarray*}
\subsection{Ordered Hamming schemes}
Martin and Stinson \cite[Section~1.3]{MARTIN-STINSON}
introduced ordered Hamming schemes as Delsarte's extension 
of length $s$ \cite[Section~2.5]{DELSARTE} of the kernel schemes.
\begin{proposition} \label{prop:extension} (Extension of length $s$)

Let $R:X \times X \to I$ be a commutative association scheme. 
Let $s$ be a positive integer, and $S_s$ a symmetric group
of degree $s$, acting on the direct product $I^s$.
We obtain a 
quotient
map $I^s \to I^s/S_s$. Then, the composition
$$
X^s \times X^s \to I^s \to I^s/S_s
$$
is a commutative association scheme.
For $\bar{i}\in I^s/S_s$, the corresponding
Hadamard primitive idempotent is
the sum of 
$$A_{i_1}\otimes\cdots \otimes A_{i_s}$$
over $(i_1,\ldots, i_s)\in I^s$ which maps to $\bar{i}$.
The set of primitive idempotents are naturally identified
with $J^s/S_s$, as follows.
For $\bar{j} \in J^s/S_s$, the corresponding
primitive idempotent is the sum of 
$$E_{j_1}\otimes\cdots \otimes E_{j_s}$$
over $(j_1,\ldots, j_s)\in J^s$ which maps to $\bar{j}$.
The corresponding eigenspace is the direct
sum of the tensors of the eigenspaces
$$
 C(X)_{j_1}\otimes\cdots \otimes C(X)_{j_s}
$$
over $(j_1,\ldots, j_s)\in J^s$ which maps to $\bar{j}$.
\end{proposition}
A proof is not given in the paper by Delsarte,
but found in an unpublished (but reachable) paper
by Godsil \cite[3.2 Corollary]{GODSIL}. Because 
he did not mention the eigenspaces, which
we need later, we shall recall a proof. 
\begin{proof}
Let
$$
X^s\times X^s \to I^s
$$
be the direct product association scheme. Its Bose-Mesner algebra
$A_{X^s}=A_{X}^{\otimes s}$ 
is generated by the $s$-fold kronecker products
of $A_i$'s for $i\in I$, which satisfies the axiom of association 
scheme. The symmetric group $S_s$ acts on $A_{X^s}$ by permutation.
We take the fixed part $A_{X^s}^{S_s}$ of $A_{X^s}$. 
Since $S_s$
preserves the two products and the transpose, fix the $i_0$ and the $j_0$, 
$A_{X^s}^{S_s}$ is a Bose-Mesner algebra. 
As a vector space, this is canonically isomorphic
to $C(I^s)^{S_s}\cong C(I^s/S_s)$, induced by $I^s \to I^s/S_s$.
For an $\bar{i}\in I^s/S_s$, its inverse image is 
the set of $(i_1,\ldots,i_s)$ mapped to $\bar{i}$,
and the set of $A_{i_1}\otimes \cdots \otimes A_{i_s}$ are mutually
disjoint Hadamard idempotent, whose sum is an Hadamard idempotent 
$A_{\bar{i}}$.
This must be primitive in $A_{X^s}^{S_s}$, since 
the Hadamard product satisfies
$A_{\bar{i}}\circ A_{\bar{i}'}=\delta_{\bar{i}\ \bar{i}'}$,
and their sum is 
the Hadamard unit $J_{A_{X^s}}$. The same argument applies
for the primitive idempotents $E_{\bar{j}}$.
The description of the corresponding eigenspace follows.
\end{proof}

This construction, when applied to the kernel schemes $X_n$, 
$I_n=\{1,2,\ldots,n\}\cup\{\infty\}$,
and $R_n:X_n\times X_n\to I_n$ given
Definition~\ref{def:kernel}), yields an ordered Hamming scheme.
\begin{definition}
The ordered Hamming scheme, denoted by $\overrightarrow{H}(s,n,v)$,
is defined as the extension of length $s$ of the kernel scheme
$\overrightarrow{k(n,v)}$.
\end{definition}

The following is easy to check.
\begin{proposition}\label{prop:pro-ordered}
The projections $X_{n+1}^s \to X_n^s$ and 
$I_{n+1}^s/S_s \to I_n^s/S_s$ coming from Proposition~\ref{prop:pro-k}
give a projective system 
$\overrightarrow{H}(s,n+1,v)\to \overrightarrow{H}(s,n,v)$
of commutative association schemes, namely, a profinite
association scheme, which
we call the pro-ordered Hamming scheme and denote by 
$\overrightarrow{H}(s,\infty,v)$.
The $\Ihat$ of this projective association schemes
is $(\N_{>0}\cup\{\infty\})^s/S_s$,
and $\Jhat$ is $\N^s/S_s$.
\end{proposition}
We shall mention on this profinite scheme 
later in the last of Section~\ref{sec:t-s}. 

\section{Delsarte theory for profinite association schemes}
\label{sec:delsarte}
Here we extend Delsarte theory introduced in \cite{DELSARTE}.
The method here follows Kurihara-Okuda
\cite{KURIHARA-OKUDA}, which generalizes Delsarte theory 
to compact homogeneous spaces. 
In this section, let $(X_\lambda, R_\lambda, I_\lambda)$
be a profinite association scheme, and $\Xhat$, $\Ihat$,
$\Jhat$, $A_{\Xhat}$ be those defined in Section~\ref{sec:def-pro}.
\subsection{Multiset and Averaging functional}
We consider a finite multi-subset $Y$ of $\Xhat$,
which means that $Y$ is a set in which finite multiplicity
of elements is allowed and taken into account.
To make the notion rigorous, we consider 
a map from a finite set $Z$ to a set $X$
$$
g: Z \to X.
$$
Then the image $Y:=g(Z)$ in $X$ has a natural
finite multiset structure, where the 
multiplicity of $y\in Y$ is the cardinality
of the fiber $\#g^{-1}(y)$.
For $S\subset X$, we use the notation
$$
\#(Y\cap S):=\#(Z\cap g^{-1}(S)).
$$
This merely means to count the number of elements
in $Y\cap S$ with taking the multiplicity into account.
Thus, we call $Y$ a ``multi-subset'' of $X$,
and use the notation 
\begin{center}
a finite multiset $Y\subset X$
\end{center}
by an abuse of language.

For any non-empty finite multiset $Y\subset \Xhat$, 
we would like to describe the notions of 
codes and designs. We begin with defining the averaging functional.
\begin{definition}\label{def:avg}
Let $Y\subset \Xhat$ be a finite multi-subset in the sense above. 
Define the averaging functional
$$
\avgY:C(\Xhat) \to \C, \quad f \mapsto 
\frac{1}{\#Y}\sum_{x\in Y}f(x) := \frac{1}{\#Z} \sum_{z \in Z} f(g(z)).
$$
\end{definition}
By (\ref{eq:commutative}), $C_c(\Jhat)\cong A_\Xhat \cong C_{lc}(\Ihat)$
holds, and by 
$$
 A_\Xhat \to C(\Xhat \times \Xhat) 
 \stackrel{\avgY\otimes \avgY}{\to} \C,
$$
$\avgY^2:=\avgY\otimes \avgY$ defines a functional on $A_\Xhat$,
on $C_c(\Jhat)$, and on $C_{lc}(\Ihat)$. For example,
for $f\in C_{lc}(\Ihat)$, we have
\begin{eqnarray*}
\avgY^2(f)&=&\frac{1}{\#Y^2}\sum_{x,y\in Y}f\circ \Rhat(x,y) \\
&=& \frac{1}{\#Y^2}\sum_{i \in \Ihat}
\sum_{x,y\in Y, \Rhat(x,y)=i}f(i) \\
&=& \frac{1}{\#Y^2}\sum_{i \in \Ihat}
\#((Y\times Y)\cap \Rhat^{-1}(i))f(i).
\end{eqnarray*}
Note that these coefficients are the inner-distribution
of $Y$ (Delsarte \cite[Section~3.1]{DELSARTE})
multiplied by $\frac{1}{\#Y}$.
Let us denote by $\C^{\oplus \Ihat}$ the
vector space whose basis is $\Ihat$, namely,
the space of finite linear combinations of the elements 
of $\Ihat$. 
We have a mapping 
\begin{equation}\label{eq:injectivity}
\C^{\oplus \Ihat} \to C_{lc}(\Ihat)^\vee, 
\end{equation}
where $\vee$ denotes the dual (i.e.\ $\Hom(C_{lc}(\Ihat),\C)$),
by the evaluation at $i$: $i \mapsto (f \mapsto f(i))$,
which is injective since only a finite number of
linear combinations appear in $\C^{\oplus \Ihat}$,
and their support can be separated by clopen subsets.
The above computation shows that 
$\avgY^2\in C_{lc}(\Ihat)^\vee$
lies in $\C^{\oplus \Ihat}$, namely,
\begin{equation}\label{eq:inner-distribution}
\avgY^2=
\sum_{i\in I}\frac{1}{\#Y^2}(\#((Y\times Y)\cap \Rhat^{-1}(i)))\cdot i
\in \C^{\oplus \Ihat}.
\end{equation}
Note that every coefficient is non-negative.
Next we compute $\avgY^2$ on $C_c(\Jhat)$.
For $j\in \Jhat$, we have the orthogonal
projector $E_j:C(\Xhat)\to C(\Xhat)_j$.
For $f\in C(\Xhat)_j$,
$$
f(x)=\int_{y\in \Xhat} E_j(x,y)f(y)d\mu(y)=(f,\overline{E_j(x,-)})_\HS.
$$
Thus, we define $\avgY^j \in C(\Xhat)_j$ by 
\begin{equation}\label{eq:avgY}
\avgY^j(-):=\frac{1}{\#Y}\sum_{x\in Y}E_j(x,-), 
\end{equation}
which represents the averaging functional in $C(\Xhat)_j$:
\begin{equation}\label{eq:repr}
(f,\overline{\avgY^j})_\HS=\avgY(f).
\end{equation}
Then
\begin{eqnarray}\label{eq:dual-distribution}
\avgY^2(E_j)&=& \frac{1}{\#Y^2}\sum_{x,y\in Y}E_j(x,y) \nonumber \\
&=& \frac{1}{\#Y}\sum_{y\in Y}\avgY^j(y) \nonumber\\
&=& (\avgY^j, \overline{\avgY^j}) =||\avgY^j||_\HS^2\geq 0.
\end{eqnarray}
These positivities make the LP method by Delsarte possible
\cite[Thorem~3.3]{DELSARTE}.

\begin{definition}\label{def:Q}
Let us denote by $Q_j\in C_{lc}(\Ihat)$ the image of $E_j \in A_{\Xhat}$.
This may be considered as a description of the canonical isomorphism
$$
Q:C_c(\Jhat) \to C_{lc}(\Ihat), \quad \delta_j \mapsto Q_j,
$$
where $\delta_j$ means the indicator function on $\Jhat$ at $j \in \Jhat$.
(Note that $E_j \in A_{\Xhat}$ corresponds to 
$\delta_j \in C_c(\Jhat)$ in Theorem~\ref{th:commutative}, see Corollary~\ref{cor:E-j}.)
Then we have an injection 
$$
\C^{\oplus \Ihat} \to C_{lc}(\Ihat)^\vee 
\stackrel{Q^\vee}{\to} C_c(\Jhat)^\vee,
$$
obtained by 
$$i \mapsto (f \mapsto f(i)) \mapsto (\delta_j \mapsto Q_j(i)).$$
\end{definition}
\begin{definition}
The above morphism $Q^\vee|_\Ihat:\C^{\oplus \Ihat} \to C_c(\Jhat)^\vee$
is called the Mac-Williams transform.
In concrete, it maps
$$
Q^\vee|_\Ihat: \sum_{i\in I} a_i\cdot i \mapsto 
\sum_{j\in \Jhat} \sum_{i\in \Ihat} a_iQ_j(i) \ev_j,
$$
where $\ev_j$ is the dual basis, i.e., $\ev_j(\delta_{j'})=\delta_{jj'}$, or equivalently, 
\[
\ev_j : C_c(\Jhat) \to \C, f \mapsto f(j).
\]
Note that the left is a finite sum, but the right may 
be an infinite sum (which causes no problem, since
the dual of a direct sum is the direct product).
\end{definition}
The following is a formal consequence:
\begin{equation}\label{eq:Mac-Williams}
Q^\vee|_\Ihat: \avgY^2 \in \C^{\oplus \Ihat} \mapsto \avgY^2 \in C_c(\Jhat)^\vee.
\end{equation}
\subsection{Codes and designs}\label{sec:code-design}
We define codes and designs.
Let $I_C \subset \Ihat$ be a subset 
which does not contain $i_0$ ($C$ for code).
Let $J_D \subset \Jhat$ be a subset
which does not contain $j_0$ ($D$ for design).  
We define a convex cone $(I_C; J_D)\subset \C^{\oplus \Ihat}$ by
\begin{eqnarray*}
(I_C;J_D)&:=& \\
&\{&\sum_{i\in I} a_i\cdot i \mid  a_i=0 \mbox{ for all but finite $i$ },\\
& & 
a_{i}\geq 0 \mbox{ for all } i\in \Ihat, \\
& & a_{i}=0 \mbox{ for all } i\in I_C, \\
& & \sum_{i\in \Ihat}a_iQ_{j}(i)\geq 0 \mbox{ for all } j\in \Jhat, \\
& & \sum_{i\in \Ihat}a_iQ_{j}(i)=0 \mbox{ for all } j\in J_D \ \  \}.
\\
\end{eqnarray*}
\begin{definition}
A non-empty finite multi-subset $Y\subset \Xhat$ is called an 
$I_C$-free-code-$J_D$-design if $\avgY^2$ lies in the cone $(I_C;J_D)$.
Furthermore, an $I_C$-free-code-$\emptyset$-design [resp.~an $\emptyset$-free-code-$J_D$-design]
is simply called an $I_C$-free-code [resp.~a  $J_D$-design].
\end{definition}
More explicitly, the $i$-component of $\avgY^2$ in $\C^{\oplus \Ihat}$ is 
\begin{equation}\label{eq:aY}
a_i(Y):=\frac{1}{\#Y^2}\#((Y\times Y)\cap \Rhat^{-1}(i)), 
\end{equation}
which is non-negative (\ref{eq:inner-distribution})
and required to be $0$ for $i\in I_C$ (namely, there is 
no pair $(x,y)\in Y \times Y$ with relation 
$\Rhat(x,y)\in I_C$), and its $\ev_j$-component 
in $C_c(\Jhat)^\vee$ is
\begin{equation}\label{eq:bY}
b_j(Y):=\avgY^2(E_j)=
||\avgY^j||_\HS^2,
\end{equation}
which is non-negative (\ref{eq:dual-distribution})
and required to be $0$ for $j\in J_D$
(namely, the $j$-component of $\avgY$ is zero
for $j\in J_D$, or equivalently, for any $f\in C(\Xhat)_j$,
$\sum_{y\in Y}f(y)=0$ for $j\in J_D$).
Note that $i_0$ is removed from $I_C$, since 
$$\frac{1}{\#Y^2}\#((Y\times Y)\cap \Rhat^{-1}(i_0))=\frac{1}{\#Y}$$
is positive, and $j_0$ is removed since $||\avgY^{j_0}||_\HS=1$
(being the operator norm of the averaging for constant functions). 
Since 
$$
\#Y\avgY^2(i_0)=1 \mbox{ and } \sum_{i\in \Ihat}\#Y\avgY^2(i)=\#Y,
$$
we may consider an LP problem: under the constraint that
$(a_i)_{i\in \Ihat}$ lies in the cone $(I_C; J_D)$ and $a_{i_0}=1$, 
maximize/minimize $\sum_{i\in \Ihat} a_i$, which 
gives an upper/lower bound on the cardinality of $Y$ which
is an $I_C$-free-code-$J_D$-design.
The following is a formal consequence of (\ref{eq:Mac-Williams}),
stating the relation with classic theory \cite[Section~3]{DELSARTE}.
\begin{theorem}
The Mac-Williams transformation $Q^\vee|_\Ihat$ maps
(\ref{eq:aY}) to (\ref{eq:bY}):
$$ 
\sum_{i\in \Ihat}a_i(Y)\cdot i \mapsto 
\sum_{j\in \Jhat}b_j(Y)\ev_j.
$$
In concrete,
$$
b_j(Y)=\sum_{i\in \Ihat}a_i(Y)Q_j(i).
$$
\end{theorem}

\begin{remark}
The meaning of the values (\ref{eq:aY}) is clear,
and called the inner distribution \cite[Section~3.1]{DELSARTE}
(up to a factor of $\#Y$),
but there, the meaning of the values
$b_j(Y)=\sum_{i\in \Ihat}a_i(Y)Q_j(i)$ is not clear, just the semi-positivity
is proved. 
Now the value $b_j(Y)=||\avg_Y^j||_\HS^2$ has a clear interpretation.
On the space $C(\Xhat)_j$, the operator norm
of $\avg_Y^j$ is $||\avg_Y^j||_\HS$, since
$$
\sup_{f\in C(\Xhat)_j\setminus \{0\}}
\frac{|\avgY^j(f)|}{||f||}
=
\sup_{f\in C(\Xhat)_j\setminus \{0\}}
\frac{(f,\overline\avg_Y^j)}{||f||}
\leq  \frac{||f||\cdot ||\avg_Y^j||}{||f||}=||\avg_Y^j||,
$$
and the equality holds for $f=\avgY^j$.
This may be interpreted as 
the worst-case error in approximating the integration 
for $C(X)_j$
by $\avg_Y$ (where the true integration value is zero for $j\neq j_0$, 
because of the orthogonality to the constant
functions $C(X)_{j_0}$), see a comprehensive book by Dick-Pillichshammer
\cite[Definition~2.10]{DICK-PILL-BOOK} on quasi-Monte Carlo integration.
Another remark: the injectivity (\ref{eq:injectivity})
implies that $Y$ and $Y'$ have the same inner distribution
$$
\#((Y\times Y)\cap \Rhat^{-1}(i))=\#((Y'\times Y') \cap \Rhat^{-1}(i))
\mbox{ for all $i\in \Ihat$}
$$
if and only if
$$
||\avgY^j||_\HS=||\avg_{Y'}^j||_\HS
\mbox{ for all $j\in \Jhat$}.
$$
\end{remark}

\subsection{$(t,m,s)$-nets and $(t,s)$-sequences}\label{sec:t-s}
The $(t,m,s)$-nets are point sets for quasi-Monte Carlo
integration, well-studied, see \cite{niederreiter:book}. 
For $(t,m,s)$-nets, the LP bound by ordered Hamming schemes 
introduced by Martin-Stinson \cite{MARTIN-STINSON}
yields strong lower bounds on the cardinality of the point set $P$,
see \cite{MARTIN-STINSON} and Bierbrauer \cite{BIERBRAUER}
for example. We do not have such an application, but 
profinite association schemes may be used to formalize
the notion of $(t,s)$-sequence \cite{niederreiter:book},
as mentioned at the last of this section. 

The purpose of this section is to reprove some results of
Marin-Stinson, and see a relation to pro-ordered Hamming schemes.
The main result Corollary~\ref{cor:main} of this section 
is proved in 
Martin-Stinson\cite[Theorem~3.4]{MARTIN-STINSON} 
by using the notion of dual 
schemes and weight enumerator polynomials.
Here we shall give a self-contained proof avoiding the use of 
these tools. Our proof matches to the
formalization in Section~\ref{sec:code-design},
relying on the average functional $\avgY$ introduced
in Definition~\ref{def:avg}.

Recall the kernel scheme in Definition~\ref{def:kernel},
where the set of alphabets is $Z_v := \Z/v$, and
$X_n=Z_v^n$, $I_n=\{1,2,\ldots,n\}\cup\{\infty \}$,
$$\bbot: X_n \to I_n$$
in Definition~\ref{def:bot}
give the kernel scheme as a translation scheme.
We shall identify a point in $X_n$
with a point in the real interval $[0,1]$, by 
\begin{equation}\label{eq:decimation}
(x_1,\ldots,x_n) \mapsto 0.x_1x_2\cdots x_n \in [0,1),
\end{equation}
where the right expression means the $v$-adic decimal expansion.
Then, we have a map
$$
X_n^s \to [0,1]^s.
$$
We shall define $(t,m,s)$-nets 
following Niederreiter~\cite[Definition~4.1]{NIED:NETS}.
An elementary interval of type
$(d_1,\ldots,d_s)$
is a subset of $[0,1]^s$ of the form
$$
\prod_{i=1}^s[a_iv^{-d_i}, (a_{i}+1)v^{-d_i}),
$$
where $a_i$ and $d_i$ are non-negative integers such
that $a_i<v^{d_i}$. Its volume is
$$
\prod_{i}^s v^{-d_i} = v^{-\sum_{i=1}^sd_i}.
$$
\begin{definition}
Let $0\leq t \leq m$ be integers. A $(t,m,s)$-net in base $v$
is a multi-subset $P\subset [0,1]^s$ consisting of $v^m$ points,
such that every elementary interval of type $(d_1,\ldots,d_s)$
with $d_1+\cdots+d_s=m-t$ (hence volume $v^{t-m}$) contains
exactly $v^t$ points of $P$.
\end{definition}

Our purpose is to state these conditions 
in terms of designs $Y\subset X_n^s$ for a finite multi-subset $Y$.
\begin{definition}
A finite multi-subset $Y$ in $X$ is said to be uniform on $X$ if $\# (Y \cap \{ x \})$ is constant for $x \in X$,
that is, $\# g^{-1}(x)$ is constant for $x \in X$
if $Y$ is defined by a map $g : Z \rightarrow X$.
\end{definition}

\begin{definition}
For an $s$-tuple of non-negative
integers $\bd=(d_1,\ldots,d_s)$ with each component less than or equal to $n$,
define
$$
X_\bd:=\prod_{i=1}^s Z_v^{d_i},
$$
and the projection
$$
\pr_\bd:X_n^s \to X_\bd
$$
by taking the left most 
$d_i$ components of the $i$-th $X_n$ in $X_n^s$, for $i=1,\ldots,s$.
\end{definition}

\begin{definition}\label{def:f-P}
A non-empty multi-subset $Y\subset X_n^s$ is $\bd$-balanced,
if the image of $Y$ in $X_\bd$ by the map $X_n^s \to X_\bd$ is uniform on $X_\bd$.
\end{definition}
\begin{proposition}\label{prop:balance}
A multi-subset $Y\subset X_n^s$ of cardinality $v^m$ gives a $(t,m,s)$-net
if and only if it is $\bd$-balanced
for any non-negative tuples $\bd=(d_1,\ldots,d_s)$
with $d_1+\cdots+d_s=m-t$.
\end{proposition}
The proof is immediate when we consider 
the mapping (\ref{eq:decimation}),
since then there is a natural one-to-one correspondence
between $X_\bd$ and the set of elementary intervals
of type
$(d_1,\ldots,d_s)$.
We may state the conditions in terms of the dual group.

\begin{lemma}\label{lem:constant}
Let $X$ be a finite abelian group, and $Y$ a finite multi-subset in $X$.
For $f\in C(X)$, define 
$$
I(f):=\avg_X(f)=\frac{1}{\#X}\sum_{x\in X}f(x).
$$
(This is the true integral of $f$ over the finite set $X$).
The following are equivalent.
\begin{enumerate}
 \item $I(f)=\avgY(f)$ holds for any $f\in C(X)$.
 \item 
 The finite multi-subset $Y$ is uniform on $X$.
 \item $\avgY(\xi)=0$ holds for any non-trivial $\xi\in \Xcheck$.
\end{enumerate}
\end{lemma}
\begin{proof}
Suppose that the first condition holds.
Let $\chi_x$ be the indicator function of $x\in X$.
The condition $I(\chi_x)=\avgY(\chi_x)$
implies $\frac{1}{\#X}=\frac{\#(Y \cap \{ x \})}{\# Y}$, hence $Y$ is uniform on $X$. 
The converse is obvious.

Suppose again the first condition. For any non-trivial $\xi$,
$I(\xi)=0$. This implies that $\avgY(\xi)=0$.
For the converse, we assume the third condition.
Note that $\Xcheck$ is a base of $C(X)$,
so it suffices to show that 
$I(\xi')=\avgY(\xi')$ holds for any $\xi' \in \Xcheck$. 
This holds for non-trivial $\xi$, and for the trivial $\xi=1$,
$I(1)=1=\avgY(1)$.
\end{proof}

\begin{lemma}
The dual $\Xcheck_\bd$ of $X_\bd$
can be identified with the set of characters
$$
\{
\xi:=(\xi_1,\ldots,\xi_s)\in \check{X_n}^s \mid
\ttop(\xi_i)\leq d_i \mbox{ for } i=1,\ldots,s
\}
$$
(see Definition \ref{def:ttop} for the notation $\ttop$).
\end{lemma}
\begin{corollary}
For a multi-subset $Y\subset X_n^s$, the composition
$Y\to X_n^s \to X_\bd$ is uniform
if and only if 
$$
\avgY(\xi)=0
$$
holds for any non-trivial $\xi\in \Xcheck_\bd$.
\end{corollary}
\begin{corollary}\label{cor:balance-dual}
A multi-subset $Y\subset X_n^s$ of cardinality $v^m$ is a $(t,m,s)$-net
if and only if
$$
\avgY(\xi)=0
$$
holds for any $\xi \in \Xcheck_\bd\setminus \{1\}$
and any non-negative
$\bd=(d_1,\ldots,d_s)$ satisfying $\sum_{i=1}^sd_i= m-t$.
\end{corollary}
Niederreiter\cite{NIED:NETS} 
and Rosenbloom-Tsfasman\cite{ROSENBLOOM-TSFASMAN}
defined the notion of NRT-weight,
see also Niederreiter and Pirsic\cite{NIEDERREITER-PIRSIC}:
\begin{definition}
Define the NRT-weight $\wt$ by the composition
$$
\wt:\Xcheck_n^s \stackrel{\ttop^s}\to J_n^s \stackrel{\sumup}{\to} \N,
$$
where $\sumup$ is a function taking the sum of the $s$
coordinates and $\ttop$ is defined in Definition~\ref{def:ttop}.
\end{definition}
\begin{theorem}\label{th:wt} Let $Y\subset X_n^s$ be a multi-subset
of cardinality $v^m$. 
Then it is a $(t,m,s)$-net if and only if for any
$\xi\in \Xcheck_n^s$ with $\wt(\xi)\leq m-t$ and $\xi\neq 1$,
$\avgY(\xi)=0$.
\end{theorem}
This follows from Corollary~\ref{cor:balance-dual}.
Following Martin-Stinson\cite{MARTIN-STINSON}, let us define 
$$
\shape: J_n^s \to J_n^s/S_s.
$$
In Proposition~\ref{prop:extension} and Lemma~\ref{lem:eigen}, 
we observed that for $\bar{j} \in J_n^s/S_s$,
the corresponding eigenspace $C(X_n^s)_{\bar{j}}$
is spanned by the characters in the inverse image of $\bar{j}$
by
$$
\Xcheck_n^s \stackrel{\ttop^s} {\to}J_n^s \to J_n^s/S_s.
$$
It is easy to see that
$\wt$ factors as 
$$
\Xcheck_n^s \stackrel{\ttop^s}{\to} 
J_n^s \stackrel{\shape}\to J_n^s/S_s
\stackrel{\height}{\to} \N,
$$
where the definition of $\height$ is clear from the context.
(We remark that $\height$ is denoted by ``height''
in Martin-Stinson\cite[P.337]{MARTIN-STINSON}, but 
we changed the notation to avoid the collision with a standard 
use of ``height'' in the context of the Young diagrams.)
Thus, 
\begin{eqnarray*}
\{\xi \in \Xcheck_n^s \mid \wt(\xi)\leq m-t\} & = &
\{\xi \in \Xcheck_n^s \mid \height(\shape(\ttop^s(\xi)))\leq m-t\}. 
\end{eqnarray*}
The condition for $Y\subset X_n^s$ to be
a $(t,m,s)$-net is that $\avgY(\xi)$ vanishes
for any nontrivial $\xi$ in the left hand side 
(Theorem~\ref{th:wt}),
which is equivalent to the vanishing in the right hand side,
namely, $\avgY(\xi)=0$ for each of $\xi \in C(X_n^s)_{\bar{j}}$
for $\bar{j}\in J_n^s/S_s\setminus \{j_0\}$, $\height(\bar{j})\leq m-t$.
This is equivalent to $\avgY^{\bar{j}}=0$ 
(the left hand side is the representation of $\avgY$
in the subspace $C(X_n^s)_{\bar{j}}$
defined in (\ref{eq:avgY}) with property (\ref{eq:repr}))
for these $\bar{j}$.

Now we proved \cite[Theorem~3.4]{MARTIN-STINSON}:
\begin{theorem}
Let $Y\subset X_n^s$ be a multi-subset
of cardinality $v^m$.
Then, $Y$ is a $(t,m,s)$-net if and only if
$\avgY^{\bar{j}}=0$ for any $\bar{j} \neq j_0$,
$\height(\bar{j})\leq m-t$.
\end{theorem}

In terms of the designs (see Section~\ref{sec:code-design}), we may state
\begin{corollary}\label{cor:main}
Define 
$$
J_D:=\{\bar{j} \in J_n^s/S_s \mid \bar{j}\neq j_0, \ \height(\bar{j})\leq m-t \}.
$$
Then, $Y\subset X_n^s$ is a $J_D$-design if and only if
it is a $(t,m,s)$-net in $X_n^s$.
\end{corollary}
By this, Delsarte's LP method works, see \cite{MARTIN-STINSON}.
The next is an example where we may use a profinite association scheme
to make a concept of $(t,s)$-sequence
inside the design theory, (see \cite[Definition~4.2]{niederreiter:book}
and \cite[Remark~2.2]{NIEDERREITER-OZBUDAK}).

Let us consider the pro-ordered Hamming scheme
$\overrightarrow{H}(s,\infty,v) = (\Xhat,\Rhat)$
defined in Proposition \ref{prop:pro-ordered}.
Note that $\Xhat = (\varprojlim X_n)^s = \varprojlim  (X_n^s) = (Z_v^{\N_{>0}})^s$.
For each $n$, we denote by $\pi_n$ the surjection from $\Xhat$ onto $X_n^s$. 
That is, 
\begin{equation}\label{eq:pi-n}
\pi_n(x^1,\dots,x^s) := ([x^1]_n,\dots,[x^s]_n)
\end{equation}
for each $(x^1,\dots,x^s) \in \Xhat = (Z_v^{\N_{>0}})^s$,
where we put $[(x_1,\dots)]_n := (x_1,\dots,x_n)$ 
for each $(x_1,\dots) \in Z_v^{\N_{>0}}$.

\begin{definition}\label{def:t-s}
A sequence of points $p_0,p_1,\ldots \in \Xhat$
is a $(t,s)$-sequence in base $v$, if 
for any integers $k$ and $m>t$, 
the point set 
$Y_{k,m} \subset X_m^s \subset [0,1)^s$ 
consisting of the $\pi_m(p_j)$ with $kv^m\leq j <(k+1)v^m$ 
(considered as a multi-set)
is a $(t,m,s)$-net
in base $v$. 
\end{definition}
For a comparison with a standard Definition~\ref{def:original-t-s}, see Remark~\ref{rem:subtle}.
We may state the condition of $(t,s)$-sequences
in terms of the designs in the pro-ordered Hamming scheme $\overrightarrow{H}(s,\infty,v)$.
Recall that $\Jhat=\N^s/S_s$, and we may define $\height$
so that the composition
$$
\N^s \to \Jhat=(\N^s/S_s) \stackrel{\height}{\to} \N
$$
maps $(j_1,\ldots,j_s) \mapsto \sum_{i=1}^s j_i$.

\begin{theorem}
In the pro-ordered Hamming scheme,
we define 
$$
J_{D_\ell}:=\{h \in (\N^s/S_s)=\Jhat \mid \height(h)\leq \ell, h\neq j_0 \}.
$$
A sequence of points $p_0,p_1,\ldots \in \Xhat$
is a $(t,s)$-sequence in base $v$, if and only if
for any integers $k$ and $m>t$,
the point set $Y_{k,m}$ 
consisting of the $\pi_{m}(p_j)$ with $kv^m\leq j <(k+1)v^m$ is 
a $J_{D_{m-t}}$-design.
\end{theorem}
In the original definition, the sequence is taken in an 
infinite set $[0,1]^s$, where infinite precision is necessary.
Because of this infiniteness, a fixed finite scheme would 
not be able to describe the notion of $(t,s)$-sequence.
Our profinite 
association scheme is a tool to deal with the infiniteness.

We close our paper with a discussion 
on the definitions of $(t,s)$-sequences.
The following is a definition
by Niederreiter-\"Ozbudak\cite[Definition~2.2]{NIEDERREITER-OZBUDAK} (to be precise, its special case),
which is a slight variant of the original Niederreiter's one: \cite[Definition~4.2]{niederreiter:book}.
\begin{definition}\label{def:original-t-s}
Let $p_0,p_1,\ldots \in [0,1]^s$ be a sequence of points,
with prescribed $v$-adic expansions.
(In other words, each $p_j$ is considered
as an element in $\Xhat$.)
It is a $(t,s)$-sequence in base $v$, if 
for any integers $k$ and $m>t$, 
the point set $P_{k,m}$ 
consisting of the $\pi_m(p_j)$ with $kv^m\leq j <(k+1)v^m$ is a $(t,m,s)$-net
in base $v$. 
\end{definition}
\begin{remark}\label{rem:subtle}
It is easy to see that Definitions~\ref{def:t-s} and \ref{def:original-t-s} are the same.
In Definition~\ref{def:original-t-s}, 
we use $\pi_m$ in (\ref{eq:pi-n}) for $p_j$, 
where $p_j$
is an element of $\Xhat$. In fact, 
$\pi_m$ is not well-defined for $[0,1]^s$,
because of the non-injectivity 
of 
$$
(Z_v^{\N_{>0}}) \to [0,1],
$$ 
namely, for base $v=2$, for example, $0.0111\cdots=0.1000\cdots$.
One way to avoid this subtle problem
may be to require the $v$-adic expansion of
a real number to be of the latter type,
namely, to avoid expansions consisting of all $(v-1)$'s
after some digit.
However, in important constructions such as in 
Niederreiter-Xing\cite{niederreiter:xing}, points with 
the former type of expansions may appear.
Thus, to define the notion of $(t,s)$-sequences
as a sequence in $[0,1]^s$ has a subtle problem.
Definition~\ref{def:t-s} would be a natural
definition of a $(t,s)$-sequence. 
In other words, the notion would be better
defined in terms of $\Xhat$ rather than $[0,1]^s$.
This is essentially 
stated in \cite[Remark~2.2]{NIEDERREITER-OZBUDAK}. 
Of course, a large part of researchers would prefer
to work in $[0,1]^s$.
\end{remark}

\bibliographystyle{plain}
\bibliography{sfmt-kanren}


\end{document}